\documentclass[10pt,a4paper]{article}
\usepackage{typearea}
\usepackage{amsfonts}
\usepackage{amssymb}
\usepackage{amsthm}
\usepackage{amscd}
\usepackage[centertags]{amsmath}
\usepackage{eucal}
\usepackage{epsfig}
\usepackage[matrix,arrow,curve]{xy}
\usepackage[small,nohug,heads=vee]{diagrams}
\usepackage{xr}
\usepackage{comment}
\diagramstyle[labelstyle=\scriptstyle]
\CompileMatrices
\typearea{12}
\author{Andreas Klein}
\title{Hamiltonian spectral invariants, symplectic spinors and Frobenius structures II}

\newcommand{\R}{\mathbb{R}}

\newcommand{\Z}{\mathbb{Z}}

\newcommand{\C}{\mathbb{C}}
\newcommand{\N}{\mathbb{N}}
\newcommand{\PI}{\mathbb{P}}

\newtheorem{theorem}{Theorem}[section]
\newtheorem{Def}[theorem]{Definition}
\newtheorem{DefLem}[theorem]{Definition/Lemma}
\newtheorem{prop}[theorem]{Proposition}
\newtheorem{lemma}[theorem]{Lemma}
\newtheorem{folg}[theorem]{Corollary}
\newtheorem{propdef}[theorem]{Proposition/Definition}

\newtheorem{asslemma}[theorem]{Assumption/Lemma}

\begin{document}
\maketitle

\begin{abstract} In this article, we continue our study of the previously \cite{kleinham} introduced concept of 'Frobenius structures' and symplectic spectral invariants in the context of symplectic spinors. By studying mainly the case of $C^1$-small Hamiltonian mappings on symplectic manifolds $M$ admitting a metaplectic structure and a parallel $\hat O(n)$-reduction of its metaplectic frame bundle we derive how the construction of 'singularly rigid' resp. 'self-dual' pairs of irreducible Frobenius structures associated to this Hamiltonian mapping $\Phi$ leads to a Hopf-module structure on the set of irreducible Frobenius structures, which we label 'Dihedral Lagrangian Hopf module'. The spectral cover of the self-dual irreducible Frobenius structure in question here realizes the graph of $\Phi$. We then generalize this specific construction of a distinguished 'dual pair' and define abstractly conditions under which 'dual pairs' associated to a given $C^1$-small Hamiltonian mapping emerge, these dual pairs are essentially pairs $(s_1, J_1), (s_2, J_2)$ of closed sections of the cotangent bundle $T^*M$ (where only $s_1$ is assumed to be $C^1$-small) and (in general singular) compatible almost complex structures on $M$ satisfying certain integrability conditions involving a certain notion of Koszul bracket, connecting different levels of the Taylor expansions of the sections $(s_1, s_2)$. In the second part of this paper, we translate these characterizing conditions for general 'dual pairs' of Frobenius structures associated to a $C^1$-small Hamiltonian system into the notion of {\it matrix factorization}. We propose an algebraic setting involving modules over certain fractional ideals of function rings on $M$ in which we prove that the set of 'dual pairs' in the above sense and the set of matrix factorizations associated to these modules stand in bijective relation. We prove, in the real-analytic case, a Riemann Roch-type theorem relating a certain Euler characteristic arising from a given matrix factorization in the above sense to (integral) cohomological data on $M$ using Cheeger-Simons-type differential characters, derived from a given pair $(s_1, J_1), (s_2, J_2)$. We propose extensions of these techniques to the case of 'geodesic convexity-smallness' of $\Phi$ and to the case of general Hamiltonian systems on $M$.

\end{abstract}

\tableofcontents

\section{Interlude}\label{interlude}

In this interlude we will discuss several concepts in an (in part) slightly informal way that are of importance in this and subsequent papers (\cite{kleinham3}, \cite{kleinlag}, \cite{klein4}).

\subsection{Semiclassics, functoriality and Frobenius structures}\label{semic}
In this section we will briefly describe several (relatively simple) examples of phenomena that can be understood as functorial relations between certain categories of Frobenius structures, that is the category with objects (in general weak, possibly non-standard) Frobenius structures $(M, \Omega, \mathcal{L}, \omega)$ in the sense of Definition \ref{frobenius} on a fixed symplectic manifold $(M,\omega)$ and with morphisms being defined as symplectic diffeomorphisms of $(M, \omega)$ that are covered by homomorphisms of the family of representations of commutative algebras $\mathcal{A}\subset {\rm End}(\mathcal{L})$ associated to the respective $(\Omega, \mathcal{L})$, arise in the context of semiclassical quantization and symplectic fixed point problems. The semiclassical viewpoint will be discussed in more detail in a subsequent paper (\cite{klein4}), the categorical viewpoint will be examined more closely in joint work with S. Krysl.
\subsubsection{Hamiltonian diffeomorphisms $C^1$-close to the identity and (self-)duality}
Let $(M, \omega)$ be a symplectic manfold of dimension $2n$ and fix a compatible almost complex structure $J$ on $TM$ and assume in the following that we have chosen a symplectic connection satisfying $\nabla J=0$ (note that we do not require $\nabla$ to be torsion-free, cf. Section \ref{spinorsconn}). Consider the symplectic manifold $(\hat M, \hat \omega)=(M\times M, (-\omega)\oplus \omega)$ together with the compatible almost complex structure $\hat J=(-J \oplus J)$. Then as is well-known (\cite{salamon}) a neighbourhood $\mathcal{N}(\Delta)$ of the diagonal $\Delta \subset M\times M$ (being Lagrangian in $\hat M$) is symplectomorphic to a neighbourhood $\mathcal{U}_0$ of the zero section $M$ of the cotangent bundle $T^*M$, the latter equipped with the standard symplectic form $\Omega_0=d\lambda$, $\lambda$ being the canonical $1$-form on $T^*M$. Note that the symplectomorphism $\phi: \mathcal{U}_0 \rightarrow \mathcal{N}(\Delta)$ is essentially determined by $J$, so
\begin{equation}\label{trivdiag}
\phi_p: (\mathcal{U}_0)_p\subset T^*_pM \rightarrow \mathcal{N}(\Delta), \ \alpha_p \mapsto {\rm exp}_{p}(\alpha_p^{*}),\ p\in M,
\end{equation}
where ${\rm exp}:T\hat M \rightarrow \hat M$ is the exponential map associated to $(\hat \omega, \hat J)$ and $(\cdot)^*:T^*M \rightarrow TM$ is the duality given by $\omega$, note that $TM\simeq T\Delta$. Assume now we have a symplectic diffeomorphism $\Phi:M \rightarrow M$ that is $C^1$-close to the identity, then its graph ${\rm gr}_{\Phi}\subset \mathcal{N}(\Delta)\subset M \times M$ is Lagrangian wrt $\hat \omega$ and its preimage under $\phi$ is a Lagangian submanifold $l=\phi^{-1}({\rm gr}_{\Phi})\subset \mathcal{U}_0\subset T^*M$ that intersects the zero section $M\subset T^*M$ exactly at the fixed points of $\Phi$. Since $l$ is $C^1$-close to the zero section $M$, it is the graph of a closed one form, so a section $s_l: M\rightarrow T^*M$ satisfying $ds_l=0$ and we are thus in the situation of Theorem \ref{genclass}, that is assuming that $c_1(M)= 0\ {\rm mod}\ 2$ allowing the choice of a metaplectic structure $P$ on $M$ and in addition assuming the reducibility of $P$ to $\hat O(n)$ (note that this condition is not neccessary, cf. Theorem \ref{genclass}), for instance by the presence of a $\nabla$-invariant Lagrangian polarization of $TM$ on $M$ (cf. Proposition \ref{higgs}) we can associate a (standard, in general singular, in general weak and non-rigid, unless $\Phi$ is Hamiltonian, cf. below) Frobenius structure $(\Omega, \mathcal{L}, \omega_0)$ over $M$ in the sense of Definition \ref{frobenius} to the symplectomorphism $\Phi:M \rightarrow M$ being $C^1$-close to the identity, that is with the notation of Proposition \ref{higgs} resp. the proof of Theorem \ref{genclass} we have for the pair $(\mathcal{L},\mathcal{A})$ of complex lines and  commutative algebras over $M$
\[
(\mathcal{L},\mathcal{A})= P^J_{L,s_l}\times_{G_L^0,\tilde \mu_2}\mathcal{A}^0_2,
\]
where $\mathcal{A}^0_2=(\C\cdot f_{0,iI}, \mathcal{A}_2(\R^{2n}, iI))$ with the notation of Section \ref{coherent}, $s_l$, together with the given $\hat O(n)$-reduction $P^J_L$ of $P$, determines the implicit section $\hat s_l:M\rightarrow P_{G/G_L^0}\simeq (\pi_P^*(T^*M) \times_{Mp(2n, \R)} P)/G_L^0$ by setting $\hat s_l(x)= ((p,\tilde s_l), p).G_L^0$ for $x \in M$ if $s_l(x)=(p, \tilde s_l(x)),\ \tilde s_l(x) \in \R^{2n},\ p\in P$, $G^0_L=\{0\}\times \hat U_L(n)$, where $\hat U_L(n)=\hat U(n)\cap \hat {\mathfrak{P}}_L\simeq \hat O(n)$ and $\hat {\mathfrak{P}}_L$ is the preimage under $\rho:Mp(2n, \R)\rightarrow Sp(2n, \R)$ of the maximally parabolic subgroup of $Sp(2n,R)$ defined by $L=\R^n\times \{0\}$. Recall also that $G=H_n\times_\rho  Mp(2n, \R)$, $\hat P_{G/G^0_L}\simeq P_G/G_L^0$ and $P_G$ is as defined in (\ref{tangent})) ($P_G\simeq \hat P_G=P\times_{Mp(2n, \R), {\rm Ad}} G$) and we have chosen a metaplectic structure $P$ over $(\hat M, \hat \omega)$. $\hat P_G$ over $M$ is then by $s_l$ reduced to $G_L^0\simeq \hat O(n)$ which is denoted by $P^J_{L,s_l}$, thus the $\hat O(n)$-reduction $P^J_{L,s_l}$ of the $G$-bundle $\hat P_G$ is fixed by $\hat s_l$, the given almost complex structure $J$ on $T^*M$ and the $\nabla$-invariant polarization of $TM$ over $M$. Note that the above constructions, that is $\mathcal{L}$ and the section $s_l:M \rightarrow T^*M$, depend not only on the symplectomorphism $\Phi$ but intrinsically on the initially chosen almost complex structure $J$. Note also that alternatively in the sense of the discussion below Theorem \ref{genclass}, we can understand $\mathcal{L}$ as the image of the global section of the bundle $\mathcal{E}_{{\rm Gr}_1(\mathcal{W})}=P^J_L\times_{G_L^0, {\rm ev}_1\circ \tilde \mu_2\circ (i, i_{\mathcal{W}})}{\rm Gr}_1(\mathcal{W})$ given by $s_l$ as described in the proof of Theorem \ref{genclass}, where the implicit embedding is here $i_{\mathcal{W}}:{\rm Gr}_1(\mathcal{W})\rightarrow \mathcal{A}_2$. Note that $P^J_L$ here is the standard reduction of $P_G\simeq \hat P_G$ to $G_L^0=\{0\}\times_\rho \hat U(n)_L$. Note also the difference of the above construction to that exhibited below Corollary \ref{curvature} of associating a Frobenius structure to a closed (thus Lagrangian) section $l:N \rightarrow T^*N=M$, where $N$ is a general $n$-dimensional manifold, which required an embedding of ${\rm im}(l)=\mathcal{L}$ into $T^*M$ and extending the associated section of $T^*M$ to a map $\tilde l: U\subset M \rightarrow T^*M|U$, where $U$ is a neighbourhood of the zero section $N$ in $T^*N=M$.\\

Assume now that the symplectomorphism $\Phi:M\rightarrow M$ is not only $C^1$-close to the identity, but also Hamiltonian, that is, it is the time-$1$-map of an (eventually time dependent) Hamiltonian function $H:M\times [0,1]\rightarrow \R$. We want to discuss whether $(\Omega, \mathcal{L}, \omega_0)$ is (weakly) rigid and self-dual in the sense of 4. resp. 6. of Definition \ref{frobenius} given this condition. Analogously, we want to examine whether the Frobenius structure $(\Omega_l, \mathcal{L}_l, \omega_0)$ associated to a closed and {\it exact} Lagrangian section $l:N \rightarrow T^*N=M$ as described above resp. below Corollary \ref{curvature}, after restriction of the implicit map $\tilde l: U\subset M \rightarrow T^*M|U$, resp. the associated map $\hat s_{\tilde l}:U\subset M\rightarrow P_{G/G_L^0}$ to $i_N:N \subset U\subset M=T^*N$, is weakly rigid and self-dual. We will see that the answer to these questions is closely related to Kirillov's method that was treated in \cite{kirillov1}, \cite{kirillov2}. Note that in both cases discussed above, there exists an (essentially uniquely determined) generating function $S:M \rightarrow \R$ with the property that $s_l=dS$ on $M$. We will see below that the set of generating functions as associated to Hamiltonian diffeomorphisms and certain $1$-forms on the product $M\times M$ stands in close relationship to the set of 'self-dual' Frobenius structures associated to a Hamiltonian diffeomorphism (being $C^1$-close to the identity). In our special cases, any of these two above Frobenius structures, given the additional exactness of the implicit Lagrangian section $s_l: M\rightarrow T^*M$, is self-dual (and thus there exist sections that satisfy a 'time dependent Schroedinger equation', where the 'time parameter' here is the parameter space $M$) after restriction of $(\Omega, \mathcal{L})$ resp. $(\Omega_l, \mathcal{L}_l)$ to a Lagrangian submanifold $L\subset M$ which plays in some sense the role of a polarization in Kirillov's theory, while the generating function $S$ determines the 'phase'. In the case of the exact Lagrangian section $l:N \rightarrow T^*N=M$ of the example discussed below Corollary \ref{curvature}, this Lagrangian submanifold is of course tautologically given by $N$, in the case of the Lagrangian section $s_l: M\rightarrow T^*M$ associated to an exact Hamiltonian diffeomorphism (which is always $C^1$-close to the identity, for the time being), there is no such Lagrangian submanifold canonically given, which leads to the neccessity of the construction of certain notions of duality. To be more precise, the triple $(\Omega, \mathcal{L}, \omega_0)$ associated to the Hamiltonian diffeomorphism $\Phi$ as described above is not rigid, but a 'time dependent Schroedinger equation' characterizing 'dual pairs' and thus implying rigidity in the sense of Definition \ref{frobenius} will be satisfied if we tensor the former with a 'dual' irreducible standard Frobenius structure $(\hat \Omega, \hat {\mathcal{L}}, \omega_0)$ to be described below, while 'self-duality' is achieved by taking the tensor product of the latter two and a further pair of irreducible standard Frobenius structures induced by the involuation that twistes the factors in $M\times M$. Note that for our choice of $S$ and $\mathcal{L}$ in the case of a $C^1$-small Hamiltonian as above, the line bundles $\mathcal{L}$ and $\hat {\mathcal{L}}$ are essentially complex-conjugates to another, while below we will construct more complicated examples of (self-)'duality' using Kirillov's theory. Thus, to describe $\hat {\mathcal{L}}$ in the given case of a $C^1$-small Hamiltonian diffeomorphism more closely, consider the Frechet-dual $\mathcal{Q}'$ of the symplectic spinor bundle $\mathcal{Q}$ over $M$ (cf. Section \ref{spinorsconn}) associated to a given metaplectic structure $P$ on $M$ and note that $\mathcal{L} \subset \mathcal{Q}$ is $1$-dimensional and locally given by elements of the form $\C \cdot f_{h, T},\ h \in \R^{2n},\ T \in \mathfrak{h}$, that is by maps $s: \hat P \rightarrow G/\tilde G$, $\tilde G=G_L^0\subset \hat U(n)$ where $\hat P$ is the $G$-extension of $P$ defined (\ref{reduction}) inducing the map $\hat s:\hat P \rightarrow  \mathcal{A}_2$ defined in (\ref{equiv}). The $L^2$ inner product on $Q$ then associates canonically a one-dimensional subspace $\mathcal{L}'\subset \mathcal{Q}'$ to $\mathcal{L}$ spanned locally on open sets $U\subset M$ by maps of the type $g \mapsto <g,\C \cdot f_{h, T}>,\ g \in \mathcal{S}(\mathbb{R}^{n})$ so that locally $\hat {\rm pr}_1\circ \hat s(x) \in \C \cdot f_{h(x), T(x)}, x \in U$. Then $\mathcal{L}'$ we define as our candidate for the 'dual' $\hat {\mathcal{L}}$. Of course the Frobenius multiplication $\Omega:TM\rightarrow {\rm End}(\hat {\mathcal{L}})$ is then just given by
\[
(\hat \Omega(X)(\hat f))(g)=\hat f(\Omega^*(X)g), \ g \in \mathcal{S}(\mathbb{R}^{n}),\ \hat f \in \mathcal{L}',
\]
where $\Omega^*=A_1-i A_2$ if we decompose $\Omega$ in the sense of Definition \ref{frobenius} ($A_1, A_2$ are formally self-adjoint acting on $L^2(\R^n)$). We will call irreducible Frobenius structures $\mathcal{L}'\subset \mathcal{Q}'$ arising in this sense still (generalized) 'standard'. Of course, when extending the Schroedinger respresentation $\pi$ of the Heisenberg group, as introduced in (\ref{expl}) to the complex numbers, the above 'dual' representation is in local frames simply the image of $f_{0,T_0}$ under an appropriate element $g \in G$ using the complex conjugate representation $\overline \pi$ instead of $\pi$. We will regard this choice of 'dual' to $\mathcal{L}$ as some sort of standard dual, closely ressembling Weil's choice of 'standard character' in \cite{weil}, the reason for the ressemblance will be clear below. On the other hand, let 
\[
\iota:M \times M\rightarrow M\times M, \quad \iota(x,y)=(y,x)
\] 
be the involution that 'switches the factors' and consider the graph of $\Phi$ in $M\times M$, composed with $\iota$, so ${\rm gr}^\iota_{\Phi}=\iota\circ {\rm gr}_{\Phi}\subset \mathcal{N}(\Delta)\subset M \times M$. Furthermore, consider $\hat J^\iota=\iota^*\hat J$. Then we can repeat the above constructions for $({\rm gr}_{\Phi}^\iota,\hat J^\iota) $ instead of $({\rm gr}_{\Phi}, \hat J)$ (note that $\hat \omega$ remains fixed on $M\times M$) and arrive at two Frobenius structures $(\Omega^\iota, \mathcal{L}^\iota, \omega_0)$ and $(\hat \Omega^\iota, \hat {\mathcal{L}}^\iota, \omega_0)$. In this situation, we can state:
\begin{theorem}\label{hopf}
Let $\Phi:M\rightarrow M$ be a $C^1$-small Hamiltonian diffeomorphism on $(M, \omega_0)$. Then we can associate to $(\Phi, M, \omega_0)$ four irreducible (in general singular), standard (in the above generalized sense) Frobenius structures $(\Omega, \mathcal{L})$, $(\hat \Omega, \hat {\mathcal{L}})$, $(\Omega^\iota, \mathcal{L}^\iota)$ and $(\hat \Omega^\iota, \hat {\mathcal{L}}^\iota)$ in the sense of Definition \ref{frobenius} so that the two irreducible Frobenius structures 
\begin{enumerate}
\item $(\Omega^e, \mathcal{L}^e)=(\Omega \otimes {\bf 1} + {\bf 1}  \otimes \hat \Omega,\ \mathcal{L}\otimes \hat {\mathcal{L}})$, 
\item $(\Omega^{e, \iota}, \mathcal{L}^{e, \iota})=(\Omega^\iota \otimes {\bf 1} + {\bf 1}  \otimes \hat \Omega^\iota,\ \mathcal{L}^\iota\otimes \hat {\mathcal{L}}^\iota)$,
\end{enumerate}
are (in general singularly) rigid and a {\it dual pair} in the sense of 6. of Definition \ref{frobenius}. Furthermore, $(\Omega^{\mathfrak{e}}, \mathcal{L}^{\mathfrak{e}})= (\Omega^e \otimes {\bf 1} + {\bf 1}  \otimes \Omega^{e, \iota},\mathcal{L}^e\otimes {\mathcal{L}}^{e,\iota})$ defines an irreducible and {\it self-dual} Frobenius structure. The spectral cover of $(\Omega^{e}, \mathcal{L}^{e})$ coincides with ${\rm im}(s_l)$ and thus intersects the zero-section $M\subset T^*M$ exactly at the fixed points of $\Phi$.
\end{theorem}
\begin{proof}
We discussed above that $\Phi$ defines a standard irreducible (in general singular) Frobenius structure $(\Omega, \mathcal{L}, \omega_0)$ on $M$ using Theorem \ref{genclass} resp. Proposition \ref{higgs}, by duality we thus get a standard irreducible Frobenius structure $(\hat \Omega, \hat {\mathcal{L}}, \omega_0)$. Let $\lambda \in \Omega^1(T^*M)$ be the canonical $1$-form on $T^*M$ and note that since $\Phi$ is Hamiltonian, $s_l^*\lambda \in \Omega^1(M)$, where $s_l:M\rightarrow T^*M$ is the section defined by $\Phi$ as described above, is exact and we have exactly $s^*_l\lambda=dS$ on $M$. Thus given a meta-unitary frame ($P$ being reduced reduced to $\hat U_L(n)$) over an open set $\tilde U\subset M$, that is a section $s_u: U\rightarrow P^J_L|U$, $s_u(p)$ defines in a nghbd $U_p$ of any $p \in \tilde U$, so $U:=U_p \subset \tilde U \subset M$ normal Darboux coordinates $\psi^\omega_{U,p}: U \rightarrow \R^{2n}\simeq T_pM$. Furthermore, identifying $T_pM\simeq T^*_pM$ by $\omega$, we recall the associated isomorphism $\phi_p:\mathcal{U}_0\cap \pi_{M}^{-1}(U)=:\mathcal{U}\subset T^*_pM \rightarrow \mathcal{N}(\Delta)$ of (\ref{trivdiag}) where $\pi_{M}:T^*M\rightarrow M$ has in the local coordinates $(q_1, q_2, p_1, p_2)$ of $\mathcal{U}\simeq U_0\subset \R^{2n}$, $U_0$ open, determined by dualizing (using $\omega$) the Darboux coordinate system on $U$ given by $\psi^\omega_{U,p}$, the form
\begin{equation}\label{difference}
\phi_p(q_1, q_2, p_1, p_2)=(x_1, y_0, y_1-y_0, x_0-x_1),
\end{equation}
where $(x_0,y_0, x_1,y_1) \in \R^{2n}$ (cf. Remark 9.24 in \cite{salamon}), furthermore we have  $\phi_p^*(\lambda|\pi_M^{-1}(U))= (y_1-y_0)dx_1+ (x_0-x_1)dy_0$, where we chose $x_1, y_0$ as coordinates on the diagonal of $\phi_p(\mathcal{U})\subset M\times M$. We claim that there are global sections $\delta \in \Gamma(\mathcal{Q}')$ and $\vartheta^\iota(s_l,S) \in \Gamma(\mathcal{L}^\iota)$ resp. $\hat \vartheta^\iota(s_l,S) \in \Gamma(\hat {\mathcal{L}}^\iota)$, so that with $\Omega^e=\Omega \otimes {\bf 1} + {\bf 1}  \otimes \hat \Omega \in {\rm End}(\mathcal{L}^e)$ we have 
\begin{equation}\label{duality2a}
\Omega^e=(\Theta^{e, \iota})^*(\frac{dz}{z}),\quad {\rm where}\quad  \Theta^{e,\iota}=<\frac{1}{i}\vartheta^\iota(s_l,S) \otimes \hat \vartheta^\iota(s_l,S), \delta\otimes \delta> \in C^\infty(M,\C^*),
\end{equation}
where the scalar product is defined here by interpreting $\hat {\mathcal{L}}$ as complex conjugate as described above, thus we can infer $\hat {\mathcal{L}}^\iota \subset \mathcal{Q}$ and setting pointwise $<a_1\otimes a_2, b_1\otimes b_2>= <a_1, b_2>\cdot <a_2, b_2>$ for $a_i, b_i \in \mathcal{Q}',\ i=1,2$. Recall that $s_l$ defines a section $\hat s_l:M\rightarrow P_{G/G_L^0}\simeq (\pi_P^*(T^*M) \times_{Mp(2n, \R)} P)/G_L^0$
and thus a section of $\mathcal{E}_{{\rm Gr}_1(\mathcal{W})}=P^J_L\times_{G_L^0, {\rm ev}_1\circ \tilde \mu_2\circ (i, i_{\mathcal{W}})}{\rm Gr}_1(\mathcal{W})$ where ${\rm Gr}_1(\mathcal{W})$ is defined by the subset of complex lines in $L^2(\R^n, \C)$ given by the set $\C \cdot f_{h, T},\ h \in H_n,\ T \in \mathfrak{h}$ and $\tilde \mu_2$ is the action of $\hat U_L(n)\simeq G_L^0\subset G$ on $\mathcal{A}_2\simeq G/G_0\cap G_U$ as introduced in the proof of Theorem \ref{genclass}. Note then that for any $p \in U\subset M$, and with the above constructed local section $\overline s_u:U\subset M\rightarrow P^J_L$, write a local representative of $\hat s_l$  as $\hat s_l(x)= ((\overline s_u(x),\tilde s_l(x)), \overline s_u(x)).G_L^0$ for $x \in U\subset M$ if $s_l(x)=[s_u(x), \tilde s_l(x)],\ \tilde s_l(x) \in \R^{2n},\ s_u(x) = \pi_R(\overline s_u(x))\in \pi_R(P_L^J),\ x \in U$, if $\pi_R:P\rightarrow R$ is the projection of $P$ onto the symplectic frame bundle $R$, we can define an assignment
\[
\vartheta(s_l, S): U \rightarrow \mathcal{E}_{{\rm Gr}_1(\mathcal{W})}, \quad (p, s_l(x), S(x)) \mapsto [\overline s_u(x), \tilde \mu_2 \left( (\tilde s_l(x), S(x), I), (\C \cdot f_{0, iI}, \mathcal{A}_2(V, J_0))\right)],\ x \in U,
\]
which by definition gives rise to the implicated global section $\vartheta(s_l, S):M\rightarrow \mathcal{L}$ (that we denote by the same symbol). Let $s_p: U\rightarrow P$ the local frame that is associated to the symplectic Darboux coordinate system induced by $s_u(p)$ ($p$ fixed) on $U$. Let $g:U \rightarrow Mp(2n, \R)$ so that $\overline s_u(x)=s_p(x).g(x),\ x\in U$. Pulling back $\vartheta(s_l, S)$ via $\phi_p^{-1}:\mathcal{U} \rightarrow \mathcal{N}(\Delta)$ to a section of $(\phi_p^{-1}|U)^*(\mathcal{L}|U)$ over $\Delta \cap \phi_p^{-1}(U)$, using (\ref{difference}) and writing $s^\Delta_l=(\phi_p)\circ s_l\circ \phi_p^{-1}|\Delta\cap \phi_p(\mathcal{U}): \Delta\cap \phi_p(\mathcal{U})\rightarrow \mathcal{N}(\Delta)$ we get setting $s^\Delta_l(x_1, y_0, 0, 0)=(x_1, y_0, y_1-y_0, x_0-x_1)$ and viewing $(x_1, y_0)$ as independent parameters on ${\rm im}(s^\Delta_l)$ resp. $U$ and thus $x_0, y_1$ as functions of $(x_1,y_0)$ (which is possible exactly because $\Phi$ is $C^1$-small) we infer for any $x \in \phi_p^{-1}(U)$ the expression
\[
\begin{split}
(\phi_p^{-1})^*&(\vartheta(s_l, S)|U)(x)= [(\phi_p^{-1})^*\overline s_u(x), \pi((\phi_p)_*\tilde s_l(x), (\phi_p^{-1})^*S(x))f_{0, iI}] \\
&=[(\phi_p^{-1})^*(s_p. g)(x),\pi((\phi_p)_*\tilde s_l(x), (\phi_p^{-1})^*S(x))f_{0, iI}]\\
&= [(\phi_p^{-1})^*(s_p)(x), e^{i\left((\phi_p^{-1})^*S(x)+ <x_0-x_1, z- (y_1-y_0)>\right)} L((\phi_p^{-1})^*g(x))(f_{0, iI}(z-{\rm pr}_{2}(\phi_p^{-1})^*\tilde s_l(x)))],
\end{split}
\]
where $z \in \R^n$, $x \in \phi_p(U)\subset \Delta$, ${\rm pr}_{2}:\R^{2n}\rightarrow \R^n,\ {\rm pr}_2(x,y)=y$. Note that $(\phi_p^{-1})^*\overline s_p:\Delta \cap \phi_p^{-1}(U) \rightarrow P$ (identifying $M \simeq \Delta$) does not take values in $P^J_L$ in general. Finally define $\delta = [s_u(x), \delta(0)\cdot e^{\frac{<{\rm pr}_{2}(\tilde s_l(x)), {\rm pr}_{2}(\tilde s_l(x))>}{2}} ] \in \Gamma(\mathcal{Q}'|U)$, note that by reduction of $P$ to $P_L^J\simeq \hat O(n)$, $\delta$ gives a globally well-defined section of $\mathcal{Q}'$ on $M$. Note that $d(\phi_p^{-1})^*S(x)|U)= (\phi_p^{-1})^*(s_l^*\lambda|U)=(y_1-y_0)dx_1+(x_0-x_1)dy_0$ (cf. \cite{salamon}, Remark 9.24) while $(\phi_p^{-1})^*(\Omega(a_j+b_j)\vartheta(s_l, S)|U)= (1+i)((y_1-y_0)_j+ i(x_0-x_1)_j)(\phi_p^{-1})^*\vartheta(s_l, S)|U), j=1, \dots, n$, where $a_i, b_i \in \R^{2n}$ as in Section \ref{coherent}. Defining a section $\hat \vartheta(s_l, S)\in \Gamma(\hat {\mathcal{L}})$ as the complex conjugate of $\vartheta(s_l, S)$, while $\vartheta^\iota(s_l,S)$ is the pullback of $\vartheta(s_l, S)$ under $\iota_\phi=\phi_p^{-1}\circ \iota\circ\phi|M: \mathcal{U}_0\cap M \rightarrow \mathcal{U}_0\cap M$, analogously for $\hat \vartheta(s_l, S)$ and $\hat \vartheta^\iota(s_l, S)$. Note that $\iota_\phi$ locally maps the independent variables $(x_1, y_0)$ on $U$ to $(x_0, y_1)$. Putting everything together and noting that with the above notations and choices the assertions of the theorem concerning duality of pairs are over any $U\subset M$ as above and in the above coordinates equivalent to the validity of 
\begin{equation} \label{cond1}
\begin{split}
(\Omega(X)\otimes {\bf 1}+{\bf 1}\otimes \hat \Omega(X))\vartheta(s_l, S)\otimes \hat \vartheta(s_l, S)|U&= (dS|U)(X)\vartheta(s_l, S)\otimes \hat \vartheta(s_l, S)|U\\
d(<x_0-x_1, y_1-y_0>)+ (\phi_p^{-1})^*(dS|U)(x)&= \iota^*(\phi_p^{-1})^*(dS|U)(x).
\end{split}
\end{equation}
for $X \in \Gamma(TU), x\in \phi_p(U)$, we arrive at the assertion (\ref{duality2a}). All remaining assertions of the theorem are proven in complete analogy.
\end{proof}
Of course we can pose the question if the above $4$-tuple of irreducible, standard Frobenius structures defines the only pair of 'dual pairs' of Frobenius structures defining a self-dual, irreducible Frobenius structure in the sense of Definition \ref{frobenius} that is associated to a given exact symplectomorphism $\Phi$ (being $C^1$-close to the identity). More precisely we aim first to classify the pairs of irreducible (in general weak) generalized standard Frobenius structures $(\Omega, \mathcal{L})$ resp. $(\hat \Omega, \hat {\mathcal{L}})$ in the sense of Definition \ref{frobenius} so that the spectral cover associated to $(\Omega^e, \mathcal{L}^e)=(\Omega \otimes {\bf 1} + {\bf 1}  \otimes \hat \Omega,\ \mathcal{L}_1\otimes \hat {\mathcal{L}_2})$ coincides with the image of the section $s_l: M\rightarrow T^*M$ with the above notation and $(\Omega, \mathcal{L})$ is (non-generalized) standard and coincides with the 'canonical' (or tautological) irreducible Frobenius structure $(\Omega, \mathcal{L})$ associated to $\Phi$ as defined above Theorem \ref{hopf} while $(\hat \Omega, \hat {\mathcal{L}})$ with $\hat {\mathcal{L}} \subset \mathcal{Q}'$ is generalized standard and irreducible (both in general singular). Moreover we demand that with $\iota: M \times M \rightarrow M\times M$ the involution defined above $(\Omega^e, \mathcal{L}^e)$ and $(\Omega^{e, \iota}, \mathcal{L}^{e,\iota})$ are a dual pair in the sense of Definition \ref{frobenius}. In a second step, we classify the irreducible, standard dual pairs in the above sense as a function of the underlying generating function $S$ (resp. the Hamiltonian diffeomorphism $\Phi$ $C^1$-close to the identity). We will see that classifying the dual pairs of Frobenius structures associated to the set of $\Phi$ in this sense leads to the problem of matrix factorizations on one hand and, as already mentioned, to Kirillov's orbit method on the other hand. To see this, note that as remarked above (\ref{cond1}) Theorem \ref{hopf} is in the coordinates introduced above locally equivalent to the specialization to $T=iI$ and $(\hat x_0, x_1, y_0, \hat y_1)=(x_0, x_1, y_0, y_1)$ of the following two local equations for $j=1, \dots, n$ (if $(\phi_p^{-1})^*(\hat \Omega(a_j+b_j)\hat \vartheta(s_l, S)|U)= (1+i)((\hat y_1-y_0)_j- i(\hat x_0-x_1)_j)(\phi_p^{-1})^*\vartheta(s_l, S)|U), j=1, \dots, n$ for $\vartheta(s_l, S) \in \Gamma(\hat {\mathcal{L}})$ while $d((\phi_p^{-1})^*S(x)|U)= (\phi_p^{-1})^*(s_l^*\lambda|U)=(y_1-y_0)dx_1+(x_0-x_1)dy_0$ and the action of $\Omega$ on $\vartheta(s_L, S)$ is as in the proof of Theorem \ref{hopf} above) while $T\in \mathfrak{h}$:
\begin{equation}\label{local}
\begin{split}
(y_1-y_0)_j- i(x_0-x_1)_j &+ (\hat y_1-y_0)_j+ \sum_{i=1}^n T_{ji}(\hat x_0-x_1)_i= 2(\phi_p^{-1})^*i_{a_j}(dS(x)|U),\\
i(y_1-y_0)_j+(x_0-x_1)_j &+ (\hat x_0-x_1)_j+ \sum_{i=1}^n \overline T_{ji}(\hat y_1-y_0)_i= 2(\phi_p^{-1})^*i_{b_j}(dS(x)|U)\\
d({\rm Im}(<\sum_{i=1}^n T_{ji}(\hat x_0-x_1)_i, \hat y_1-y_0>) &+ <x_0-x_1, y_1-y_0>)/2+ (\phi_p^{-1})^*(dS|U)(x)= \iota^*(\phi_p^{-1})^*(dS|U)(x).
\end{split}
\end{equation}
where $\overline T_{ji}$ denote the components of the complex conjugated matrix $\overline T$ if $T \in \mathfrak{h}$, while in the second line the evaluation on $b_j, j=1,\dots, n$ ($a_i,b_i, i=1, \dots, n$ are here as defined in Section \ref{coherent}) entails a relative sign change between the first two and the subsequent summands on the left hand side since $\Omega^*(J(\cdot))=-i \Omega^*(\cdot)$.
We will see that this local condition is sufficient and neccessary to determine the pairs of generalized standard, irreducible Frobenius structures that give dual pairs in the above described sense. Assume we have given the $\tilde G = \hat U(n)\subset G$-reduction $P_{\tilde G}$ of $P$ representing tautologously the standard $\tilde G$-reduction of the $G$-bundle $\hat P$ as given in (\ref{reduction}). Assume in addition we have given a global section $\hat s_l: M\rightarrow \hat P_{G/\tilde G}\simeq P_{G/\tilde G}\simeq \hat P/\tilde G$ associated to a closed section $s_l: M\rightarrow T^*M$ resp. a Hamiltonian diffeomorphism $\Phi$ and the choice of the $\omega$-compatible almost complex structure inducing $P_{\tilde G}$, where $P_G= (\pi_P^*(T^*M) \times_{Mp(2n, \R)} P)\simeq \hat P_G=P\times_{Mp(2n, \R), {\rm Ad}} G$ as explained below (\ref{tangentaction}). As above, using these isomorphisms write local representants of $\hat s_l$ as $\hat s_l(x)= ((\overline s_u(x),\tilde s_l(x)), \overline s_u(x)).\tilde G$ for $x \in U\subset M$ if $s_l(x)=[s_u(x), \tilde s_l(x)],\ \tilde s_l(x) \in \R^{2n},\ s_u(x)= \pi_R(\overline s_u(x))\in \pi_R(P_{\tilde G}),\ x \in U$, if $\pi_R:P\rightarrow R$ is the projection of $P$ onto the symplectic frame bundle $R$. Consider now $\phi: \mathcal{U}_0 \rightarrow \mathcal{N}(\Delta)$ as defined in (\ref{trivdiag}) and consider $P_G^\Delta= (\phi_p^{-1})^*P_G$, then $(\phi_p^{-1})^*\hat s_l: \phi_p(U) \rightarrow P_G^\Delta/\tilde G$ and if $s_p: U\subset M \rightarrow P$ is the symplectic Darboux frame induced by $s_u(p)$ for a fixed $p \in U$ on $U$ as in the proof of Theorem \ref{hopf}, we have with $P^\Delta= (\phi_p^{-1})^*P$ (analogously for $R$) that $(\phi_p^{-1})^*s_l(x)=[(s^\Delta_p(x), \tilde s_l(x).g(x)],\ \tilde s_l(x) \in \R^{2n},\ s^\Delta_u(x):=(\phi_p^{-1})^*s_u(x)= \pi_{R^\Delta}((\phi_p^{-1})^*\overline s_u(x))\in P^\Delta,\ x \in U$ analogously defined $s^\Delta_p(x)$ and $s^\Delta_u(x)=s^\Delta_p(x).g(x), x \in U$ for $g:U \subset M \rightarrow P^\Delta$. Denote by $\pi_{Mp}: G/\tilde G\rightarrow Mp(2n, \R)/\tilde G$ the projection to the subquotient. \\
Let now $\hat s_2: M\rightarrow \hat P_{G/\tilde G}\simeq P_{G/\tilde G}$ be another section of the bundle $\hat P_{G/\tilde G}$, with associated closed section $s_2(x):M \rightarrow T^*M$ given by the composition of $\hat s_2$ with $\tilde {\rm pr}_1: P_{G/\tilde G}\rightarrow T^*M$ (the latter as described above Theorem \ref{genclass}), then there exists an equivariant map $\overline T:P\rightarrow Mp(2n, \R)/\hat U(n)$ so that $\pi_{Mp}(\hat s_l).\overline T(p)=\pi_{Mp}(\hat s_2),\ p\in P$ if we identify $\hat s_l, \hat s_2:P_G\rightarrow G/\tilde G$ and $\pi_{Mp}(\hat s_l)$ with the identity section in $P/\hat U(n)=P\times Mp(2n, \R)/\hat U(n)$. Note also that we identify $\mathfrak{h}\simeq Mp(2n, \R)/\hat U(n)\simeq Sp(2n, \R)/U(n)$, writing $T:P\rightarrow Sp(2n, \R)/U(n)$ for the image of $\overline T$. Note that $\overline T$ is represented by projecting an automorphism (a gauge transformation, thus covering the identity on $M$) $g:P_G\rightarrow P_G$ satisfying $\hat s_l.g=\hat s_2$ and being represented by $\tilde g: P_G\rightarrow G$ to $Mp(2n, \R)/\tilde G$ and restricting its image to $P$. Then we can locally write with $s_u$ as above $\hat s_2(x)= ((\overline s_u(x),\tilde s_2(x)), \overline s_u(x)).\tilde G$, where $0=(\tilde s_2-T.\tilde s_2^0).\tilde G: U\subset M \rightarrow \R^{2n}$ if $\tilde s_2^0: U\rightarrow \R^{2n}$ so that $s_2= [s_u^0(x), \tilde s_2^0(x)], x  \in U$, where $s_u^0: U\subset M\rightarrow P_{\tilde G,2}:=s_2^*(\tilde E_G)$ (notation of the proof of Theorem \ref{genclass}) and we can thus write locally $s_2(x)=[s_u(x), T(\tilde s_2^0(x))],\ \tilde s_2^0(x) \in \R^{2n}$. Pulling back the local expressions for $\hat s_2, s_2$ to $\Delta$ and using the Darboux frames $s^\Delta_p(x) \in P^\Delta$ as above, we have nearly proven the following theorem, which we postpone after having established the following Lemma. Note that here and in the following in this article, we assume that $P$ is reducible to $\hat O(n)\subset \hat U(n)\subset MP(2n,\R)$, the corresponding reduction denoted as before by $P_{\tilde  G}$ with $\tilde G=\hat O(n)$. This is for instance the case if $M$ is equipped with a global Lagrangian distribution $\Lambda\subset TM$ which we will occassionally, but not throughout, also assume to be invariant by $\nabla$ thus giving a reduction of the associated connection $Z$ on $P$ to $\tilde Z:TP_{\tilde G}\rightarrow \widehat {\mathfrak{o}(n)}$, where $\widehat {\mathfrak{o}(n)}$ is the Lie algebra of $\hat O(n)$.
\begin{DefLem}\label{koszul}
Let $s_1, s_2 \in \Gamma(T^*M)$ and $s_1^*, s_2^* \in \Gamma (TM)$ be their $\omega$-duals. Let $\Lambda_{\tilde G} \in \Gamma({\rm Lag}(TM, \omega),U_i)$ be a global section of the Lagrangian Grassmannian of $(TM, \omega)$ inducing the chosen $\hat O(n)$-reduction $P_{\tilde G}$ of $P$ reducing the given $\hat U(n)$ reduction of $P$ induced by $J$. Then the $\C$-bilinear map, dependent on the chosen $\omega$-compatible structure $J$
\[
[\cdot, \cdot]_{J, \Lambda_{\tilde G}}: \Gamma(T^*M)^2 \rightarrow \Gamma(T^*M), \quad (s_1, s_2) \mapsto [s_1, s_2]_{J,  \Lambda_{\tilde G}}= \ d\left(\omega({\rm pr}_{\Lambda_{\tilde G}}(s_1^*), {\rm pr}_{J\Lambda_{\tilde G}}(s_2^*))\right),
\]
where ${\rm pr}_{\Lambda}: T^*U\rightarrow T^*U$ for any $\Lambda \in \Gamma({\rm Lag}(TM, \omega),M)$ denotes the projection onto $\Lambda$ according to the direct sum decomposition $TU=\Lambda \oplus J\Lambda$, is well-defined. We will call $[\cdot, \cdot]_{J, \Lambda_{\tilde G}}$ the (symplectic) Koszul bracket on $M$ associated to $J$ and $\Lambda_{\tilde G}$ (and frequently notationally suppress the dependency on $\Lambda_{\tilde G}$ in the following).
\end{DefLem}
\begin{proof}
Since $U(n)$ acts transitively on the set ${\rm Lag}(\R^{2n}, \omega_0)$ with isotropy group $U(n)$, we conclude that the fibre bundle ${\rm Lag}(TM, \omega)=R^J\times_{U(n)} {\rm Lag}(\R^{2n}, \omega_0)$, where $R^J$ is the given $U(n)$-reduction of the symplectic frame bundle $R$ over $M$, is isomorphic to  $R^J\times_{U(n)} U(n)/O(n)$. Now since $U(n)$ is the symmetry group of the pair $(\omega, J)$, we see that it is also the symmetry group of any local bracket $(s_1, s_2)\mapsto d(\rho_i\omega({\rm pr}_{\Lambda_i}(s_1^*), {\rm pr}_{J\Lambda_i}(s_2^*)), i \in I$ and acts transitively on the set of these local brackets, so all these coincide, which was the claim.
\end{proof}
\begin{theorem}\label{theorem2}
Given $(M, \omega)$ and a compatible almost complex structure $J$ and associated metric $g$, $P$ a metaplectic structure on $M$, to each pair of sections $\hat s_l, \hat s_2: M\rightarrow P_{G/\tilde G}$, where the associated $s_l: M\rightarrow T^*M$ is exact with $S:M\rightarrow \R$ its primitive and $s_2:M \rightarrow \R$ associated to $\hat s_2$ is closed, (notation as above) we can associate a map $T:P\rightarrow Sp(2n, \R)/U(n) \simeq \mathfrak{h}$ so that if $[\cdot, \cdot]_J$ is the symplectic Koszul bracket on $\Gamma(T^*M)$ associated to $J$ as defined in Definition \ref{koszul} while $[\cdot, \cdot]_{J_T}$ is the Koszul bracket corresponding to $J_T$ ($J_T$ is defined below (\ref{global})), we have that if $(\Omega, \mathcal{L})$ is the canonical standard irreducible Frobenius structure associated to $\hat s_l$ as above, $(\hat \Omega, \hat {\mathcal{L}})$ is the dual of the canonical standard irreducible Frobenius structure associated to $\hat s_2$ and $\iota$ is the involution defined above, then the pair $(\Omega^e, \mathcal{L}^e)$ and $(\Omega^{e, \iota}, \mathcal{L}^{e,\iota})$ associated as above define a dual pair in the sense of Definition \ref{frobenius} so that the spectral cover of $(\Omega^{e}, \mathcal{L}^{e})$ coincides with ${\rm im}(s_l)=dS$ if and only if 
\begin{equation}\label{global}
\begin{split}
s_l(\alpha_J^+(\cdot))&+ s_2(\alpha_{J_T}^-(\cdot))= 2dS(\cdot) \in \Omega^1(M, \C)\\
\frac{1}{2}([s_2, s_2]_{J_T} &+  [s_l, s_l]_J)(\cdot)+dS(\cdot)= \iota^*(dS)(\cdot) \in \Omega^1(M, \C),
\end{split}
\end{equation}
where $J_T$ here and above is the almost complex structure on $M$ associated to the $\hat U(n)$-reduction of $P$ given by $\pi_{Mp}\circ \hat s_2$ and $\alpha^\pm_{J}, \alpha^\pm_{J_T}$ are the projections onto the $\mp i$-eigenspaces of $J, J_T$ on $T_{\C}M$, understood as $\R$-linear injections $TM\hookrightarrow T_{\C}M$. Furthermore, the condition (\ref{global}) corresponds to the local condition (\ref{local}) in the coordinates on $\Delta$ introduced above.
\end{theorem}
{\it Remark.} In this special situation, that is under the requirement that the spectral cover of $(\Omega^{e}, \mathcal{L}^{e})$ coincides with ${\rm im}(s_l)=dS$, we see that the first condition means (while dualizing by $\omega$) that the intersection $E\subset \Gamma(T_{\C}M)$ of the affine space $-dS^* + {\rm im}\ (\alpha_J^+)\subset \Gamma(T_{C}M)$ and the linear space ${\rm im}\ \alpha_{J_T}^-\subset \Gamma(T_{\C}M)$ is non-empty when interpreting $\alpha_J^+, \alpha_{J_T}^-$ as linear maps $\Gamma(TM) \rightarrow \Gamma(T_{\C}M)$ and there exists an element $s \in \Gamma(TM)$ of the preimage of $E$ under $\alpha_J^+$ which is exactly $dS^*$.
\begin{proof}
Note that relative to the frame $\overline s_u: U\subset M\rightarrow P_{\tilde G}$ and corresponding trivialization $\phi_u(x): T_xU \rightarrow \R^{2n}, x \in U$, we can write the above two conditions (\ref{global}) using $\phi_u$ since $Ad((\phi_u)_*^{-1})(\alpha_J^\pm) (a_j)= a_j - iJ_0a_j$ where $J_0a_j=b_j$ and $Ad((\phi_u)_*^{-1})(\alpha_{J_T}^\pm) (a_j)= a_j - iJ_T(a_j)$ where $a_j-iJ_T(a_j) =a_i -\sum_{i=1}^n T_{ji}b_i$ for $j=1, \dots,n$ following the above comments as  
\begin{equation}\label{local2}
\begin{split}
(\tilde s_l)_{j} -i(\tilde s_l)_{j+n} &+ (\tilde s_2)_{j}+ \sum_{i=1}^n T_{ji} (\tilde s_2)_{i+n}  = 2dS(a_j) \in C^\infty(U, \C)\\
i(\tilde s_l)_{j}+ (\tilde s_l)_{j+n} &+ (\tilde s_2)_{j+n}+ \sum_{i=1}^n \overline T_{ji} (\tilde s_2)_{i}  = 2dS(b_j) \in C^\infty(U, \C)\\
d({\rm Im}(\sum_{j=1}^n<(\tilde s_2)_{j+n},\sum_{i=1}^n T_{ji} (\tilde s_2)_{i}>) &+  {\rm Im}(\sum_{j=1}^n<(\tilde s_l)_{j+n}, i (\tilde s_l)_{j}>))(\cdot)/2+dS(\cdot)= \iota^*(dS)(\cdot) \in \Omega^1(U, \C),
\end{split}
\end{equation}
where we noted that $J_T$ preserves $< \cdot, \cdot>_T$ and $<\cdot,\cdot>$ in the last line simply means the scalar product on $\R$. But pulling back the latter expressions via $\phi_p^{-1}$ for a fixed $p \in U$ gives exactly the local expressions relative to Darboux frames (\ref{local}). That these local conditions are neccessary and, given global well-definedness of $(\Omega, \mathcal{L})$, $(\hat \Omega, \hat {\mathcal{L}})$ and (\ref{local2}), also sufficient for the pair $(\Omega^e, \mathcal{L}^e)$ and $(\Omega^{e, \iota}, \mathcal{L}^{e,\iota})$ to define a dual pair in the sense of Definition \ref{frobenius}, is proven in complete analogy to the proof of Theorem \ref{hopf}. Note that the first two lines in (\ref{local2}) are equivalent to the condition that the spectral cover of $(\Omega^{e}, \mathcal{L}^{e})$ coincides with ${\rm im}(s_l)=dS$ while the third line is, given that prior condition is satisfied, equivalent to $(\Omega^e, \mathcal{L}^e)$ and $(\Omega^{e, \iota}, \mathcal{L}^{e,\iota})$ defining a dual pair in the sense of Definition \ref{frobenius}. Note finally that given (\ref{global}) with an appropriate $S:M\rightarrow \R$, there is is a $C^1$-small Hamiltonian mapping $\Phi:M\rightarrow M$, so that $dS$ is the image of the graph of $\Phi$ in $M\times M$ under $\phi^{-1}: \mathcal{N}(\Delta) \subset M\times M \rightarrow \mathcal{U}_0$, see the remark below this proof.
\end{proof}
{\it Remark.} Note that the second condition in (\ref{global}) or equivalently (\ref{local2}) resp. (\ref{local}) reflects the fact that $s_l: M\rightarrow T^*M$ is by assumption not only exact (closed and sufficiently $C^1$-small), but is given by the image of the graph of a symplectic diffeomorphism $\Phi:M\rightarrow M$ under the diffeomorphism $\phi^{-1}: \mathcal{N}(\Delta)\subset M\times M \rightarrow \mathcal{U}_0$. If either the bijectivity or the $C^1$-smallness of $\Phi$ or the fact that $\Phi$ is a symplectomorphism is dropped, the second condition of (\ref{global}) ceases to hold, so the condition reflects a certain special symmetry in $s_l$ resp. its primitive, thus $S$. Note further that the 'dual partner' $s_2:M \rightarrow \R$ as associated to $\hat s_2: M\rightarrow P_{G/\tilde G}$ in general neither is assumed to (or is concluded to) have this symmetry nor to be $C^1$-small.\\
Note also that the second condition in (\ref{global}) or equivalently (\ref{local2}) resp. (\ref{local}) has to be read as a data that connects different levels of Taylor expansions of a given function, here $S$, on $M$, while the $0$-th degree of the expansion remains locally invisible but is of course the reason why the local equations 'assemble' to the global condition (\ref{global}). Analogously, we want to argue below how to derive an in some sense inverse result: any pair of 'dual' Frobenius structures in the sense of Theorem \ref{theorem2} defines, under an appropriate condition of 'geodesic convexity-smallness' (not necessarily '$C^1$-smallness') a Hamiltonian system $C^1$-close to the identity. We expect these results to generalize to the case of general Hamiltonian systems without the 'geodesic convexity-smallness', in which case we, up to now, only know the corresponding result to Theorem \ref{hopf} (cf. \cite{kleinham3}), that is any Hamiltonian defines a pair of 'dual' Frobenius structures on $M$ whose spectral cover intersects $M$ exactly in the fixed points of its time $1$-flow.

\subsubsection{Dihedral Lagrangian Hopf modules}

We will in this section study a Hodge module with a certain kind of additional structure, which we labelled 'Dihedral Lagrangian Hopf module'. A natural generalization of this concept would be probably that of a 'cyclic Lagrangian' Hopf module, defined in an appopriate sense, but we will not pursue this path in this paper. In any case, the construction bears ressemblance to certain structures in Class Field theory related to cyclic factor sets (cf. \cite{weil}, Chapter IX, §4) and bears probably connections to the construction of the cyclic cohomology of Hopf algebras as for instance reflected in the article (\cite{menichi}). We will prove that the structure appearing in Theorem \ref{hopf} is a Dihedral Lagrangian Hopf module in this sense.\\
Let $B$ be an algebra, that is a vector space over a field $k$ with an associative multiplication $\mu: B\otimes B \rightarrow B$ and unit $\eta: k\rightarrow B$ written $\eta(1)=1_B$ so that $\mu\circ (\eta\otimes Id)= \mu\circ (Id\otimes \eta)=Id$ where $k \otimes B\simeq B\simeq B\otimes k$ are identified. We will follow in this general discussion of Hopf algebras resp. Hopf modules essentially (\cite{hopf}) and kindly ask the reader to refer for diagrams, (proofs of) theorems and examples to loc. cit. A bialgebra is then an algebra $(B, \mu, \eta)$ together with a coassociative comultiplication $\Delta:B \rightarrow B\otimes B$ and a counit $\epsilon:B \rightarrow k$ so that thus
\[
(\Delta \otimes {\rm Id})\circ \Delta= ({\rm Id} \otimes \Delta)\circ \Delta, \ (\epsilon\otimes {\rm Id})\circ \Delta= (Id\otimes \epsilon)\circ \Delta={\rm Id},
\]
where again $k \otimes B\simeq B\simeq B\otimes k$. Let now $(A, \mu_A, \eta_A)$ be an algebra and $(B, \mu_B, \eta_B, \Delta, \epsilon)$ be a bialgebra. Let $f,g \in Hom_k(B,A)$ be $k$-algebra homomorphisms. We introduce a product on the set of such homomorphisms as follows
\begin{equation}\label{conv}
\begin{split}
\star: Hom_k(B,A) \times Hom_k(B,A) \rightarrow Hom_k(B,A)\\
(f,g) \mapsto \mu_A \circ (f\otimes g)\circ \mu_B.
\end{split}
\end{equation}
It an be shown (cf. loc. cit.,\ Lemma I.9) that $(Hom_k(B,A), \star, \eta_A\circ \epsilon_B)$ defines an algebra, $\star$ is called the convolution product. We further define morphisms of bialgebras as those algebra morphisms $f: (B, \mu_B, \eta_B, \Delta, \epsilon)\rightarrow (B', \mu_B', \eta_B', \Delta', \epsilon')$ that intertwine comultiplication and counit, that is
\[
\Delta '\circ f= f \otimes f\circ \Delta: B \rightarrow B'\otimes B, \quad \epsilon'\circ f=\epsilon.
\]
We are now in the position to define the notion of Hopf algebra:
\begin{Def}
A Hopf algebra is a bialgebra $H$ together with a linear map $S:H \rightarrow H$ (called the antipode) that satisfies
\begin{equation}\label{hopfalg}
S\star {\rm Id}_H =\eta \circ \epsilon= {\rm Id}_H \star S.
\end{equation}
\end{Def}
{\it Remark.} Since $S$ is by definition the inverse of ${\rm Id}_H$ wrt the convolution product in $Hom_k(H,H)$, it is unique, as inverses in algebras are unique.\\
Let now $B$ be a bialgebra and $\tau: B\otimes B \rightarrow B\otimes B$ be the linear involution so that $\tau(a\otimes b)=b \otimes a,\ a,b \in B$. Let $\Delta^{op}$ be defined by $\Delta^{op}=\tau \circ \Delta$. $B$ is called {\it cocommutative} if $\Delta^{op}=\Delta$. Let in the same situation be $\mu^{op}_B=\mu_B\circ \tau$. We have the following Proposition (cf. \cite{hopf}, Prop. I.26)
\begin{prop}\label{dualityhopf}
Let $H$ be a Hopf algebra, then $S: (H, \mu, \eta, \Delta, \epsilon)\rightarrow (H, \mu^{op}, \eta, \Delta^{op}, \epsilon)$ is a morphism of bilagebras, in other words we have for any $x,y \in H$
\[
S(xy)=S(y)S(x), \  s(1)=1, \ \epsilon(S(x))=\epsilon\ {\rm and}\ \Delta^{op}\circ S= (S\otimes S)\circ \Delta.
\]
\end{prop}
Note further that a morphism $f$ of Hopf algebras $H$ and $H'$, which is by definition a morphism of the underlying bialgebras, automatically satsifies $f \circ S'= S\circ f$. Note also that given a finite dimensional Hopf algebra $H$, then its algebraic dual $H^*$ carries a canonical bialgebra structure and furthermore, $B^*$ is a Hopf algebra with antipode $S^*$, the transpose of $S$. In this case, the canonical isomorphism $i:H \rightarrow H^{**}$ is in fact also an isomorphism of Hopf algebras (loc. cit., Prop. I. 12, 15, 28).\\
Recall that a (left-)module over an algebra $B$ is a vector space $M$ endowed with a map $\mu_M: B \times M \rightarrow M$ so that $\mu_M$ is associative, that is $\mu_M \circ (\mu \circ Id)= \mu_M \circ (Id\circ \mu_M):B \times B\times M\rightarrow M$ and one has $\mu_M\otimes(\eta\otimes Id)=Id$ if we again identify $k \otimes M \simeq M$. Similarly, a right module over $B$ is a vector space $M$ endowed with a map $\mu_M: M \times B \rightarrow M$ that $\mu_M$ is associative in the same sense and one has $\mu_M\otimes(Id\otimes \eta)=Id$. A bicomodule over $B$ is a module that is both a right and left-module and the left and rght actions commute. The following notion of {\it left/right comodule} in a sense dualizes these notions.
\begin{Def}\label{comodule}
Let $B$ be a bialgebra. A left comodule is a pair $(M, \rho_M)$ where $M$ is a vector space and $\rho_M$ a linear map $\rho_M: M \rightarrow B \otimes M$ satisfying coassociativity $(\Delta_M\otimes Id)\circ\rho_M= (Id\otimes \rho_M)\circ\rho_M$ and being compatible with the counit of $B$ in the sense of $(\epsilon \otimes Id)\circ \rho_M=Id$ where again we identify $k \otimes M\simeq M$. A right comodule is similarly a pair $(M, \rho_M)$ where $\rho_M: M \rightarrow M\otimes B$ and $(\rho_M\otimes Id)\circ\rho_M= (Id\otimes \Delta_M)\circ\rho_M$ and $(Id\otimes \epsilon)\circ \rho_M=Id$. A left and right-comodule whose structure maps commute will be called a bicomodule.
\end{Def} 
{\it Remarks.} Note that $B$ is a bicomodule over itself, with structure map given by $\Delta_B$. Note that if $M$ is a left $B$-module, then $M$ is also a $B$ right module by the action of $S$, that is $\mu_M(m, b)=S(b)m,\ \forall b \in B,\ m \in M$ defines a right action of $B$ on $M$. This follows from Proposition \ref{dualityhopf}. Furthermore, $M$ being a left $B$-module implies that its algebraic dual, $M^*$, is a right $B$-module by setting $\mu_M(m^*, b)=m^*(bm),\ \forall b \in B,\ m^* \in M,\ m \in M$. Hence $M^*$ is a left $B$-module via $S$.\\
Let $M, N$ be two left $H$-modules over the Hopf algebra (or bialgebra) $H$. Then $M\otimes N$ is a left $H$-module via the left action of $\Delta$, that is
\begin{equation}\label{diagact}
\mu_{M\otimes N}(h, m\otimes n)= ((\mu_M(\cdot,m) \otimes \mu_N(\cdot, n))\circ \Delta(h),\ \forall m \in N,\ m \in M,\ h \in H.
\end{equation}
Dually, let $M, N$ be two left comodules over a Hopf algebra (or bialgebra) $H$, then $M\otimes N$ is a left $H$-comodule via the diagonal coaction 
\begin{equation}\label{codiagact}
\rho_{M \otimes N}(m \otimes n)=(\mu_M \otimes Id)\circ (Id \otimes \mu_N \otimes Id) (\rho_{M} \otimes \rho_N)(m \otimes n),
\end{equation}
or in Sweedler notation (cf. loc cit) $\rho_{M \otimes N}(m \otimes n)= \sum_{(m),(n)}m_{(-1)} n_{(-1)} m_{(0)}\otimes n_{(0)}$. Finally note that a morphism of left comodules $M, N$ is a linear map $f: M \rightarrow N$ so that $(Id \otimes f) \circ \rho_M= \rho_N \circ f$, similarly for right comodules and bicomodules. We finally arrive at the definition of {\it Hopf module}. We emphasize that (\ref{diagact}) as well as the formula for the diagonal coaction make sense for (co-)modules $M, N$ which are defined over distinct (non-isomorphic) Hopf algebras $H_1, H_2$ whose underlying bialgebras are isomorphic ($H\simeq H_1\simeq H_2$ as bialalgebras) since the explicit form of the respective antipodes $S_1, S_2$ does not enter the definitions.
\begin{Def}
Let $H$ be a Hopf algebra. A left Hopf module over $H$ is a left $H$-module $M$ that is also a left $H$-comodule whose characteristic map $\rho_M: M \rightarrow H \otimes M$ is a morphism of $H$-module structures, where the $H$-module structure on $H \otimes M$ is the diagonal structure alluded to above. A morphism of left Hopf modules is a morphism of left $H$-modules that is also a morphism of left $H$-comodules, as defined above. The definition of right Hopf module and Hopf bimodule are similar.
\end{Def}
{\it Remark.} Note that $H$ is a (left and right) Hopf module over itself with coaction $\Delta$. On the other hand, given a left $H$-module $M$, $H\otimes M$ is a left Hopf module over $H$ with $H$-action the diagonal left action of $H$ on $H \otimes M$ (cf. (\ref{diagact})) and with left coaction $\Delta \otimes Id$. It is the latter Hopf module structure over a left $H$-module $M$ that we will exploit in the following.

Consider now the following structure: let $H^*$ be the symplectic Clifford algebra, $H_{\C}^*$ its complexification, which is an algebra over $\C$ so that $H_{\C}^1$ has the structure of a symplectic vector space $(H_{\C}^1, \omega)$ over $\C$. Let $L$ be a complex Lagrangian subspace $L\subset H_{\C}^1$ with the property that the subalgebra over $\R$ $L^*\subset H_{\C}^*$ generated by the elements of $L$ is commutative. Then $B=(L^*, \mu, \eta)$ is an (infinite-dimensional) graded algebra over $\R$ whose degree one part $L=L^1$ carries the structure of a Lagrangian subspace of a symplectic vector space $(H^1, \omega)$ over $\C$. Assume that there is a canonical $\R$-linear isomorphism $\phi:H^1\simeq L$, then $\phi$ inherits on the real vectorspace $H^*$ the structure of the commutative algebra $B$, this structure carries over to $H_{\C}^*$ complex-linearly. We will denote this latter commutative algebra structure on $H_{\C}^*$ in the following by $H_c^*$. Let further be $\hat H=(H_c^*, \mu, \eta, \Delta, \epsilon, \hat S)$ be a (commutative and cocommutative) Hopf algebra with underlying algebra structure $B$ and assume that $S_0, S_1$ are graded bialgebra homomorphisms $S_i: (H_c^*, \mu, \eta, \Delta, \epsilon)\rightarrow (H_c^*, \mu^{op}, \eta, \Delta^{op}, \epsilon),\ i \in \{0,1\}$ with the notation of Proposition \ref{dualityhopf}, i.e. the $S_i,\ i=0,1$ are $B$-antihomomorphisms, and that $S_2$ is a graded bialgebra homomorphism $S_2\in Aut(H_c^*, \mu, \eta, \Delta, \epsilon)$, i.e. $S_2$ is a $B$-homomorphism. Consider two left Hopf modules $(M_0, \mu_{M_0}, \rho_{M_0})$ and $(M_1, \mu_{M_1}, \rho_{M_1})$ over $\hat H$ that are both assumed to be $B$-modules. Assume furthermore the following:
\begin{enumerate}
\item $\Delta(H^1_c)\subset L \otimes Id_{\hat H_c^1} \oplus Id_{H_c^1} \otimes L$, where $L\subset H_{\C}^1$ is the complex {\it Lagrangian} subspace wrt $\omega$ on $H_{\C}^*$ generating the commutative subalgebra $B\subset H_{\C}^*$ alluded to above.
\item Assume that the $B$-antihomomorphisms $\hat S, S_0, S_1 \in {\rm Aut}_{\C}(B,B^{op})$ and the $B$-homomorphism $S_2 \in {\rm Aut}_{\C}(B)$ introduced above satisfy the relations
\[
\hat S^2=S_0^2= S_1^2= S_2^2= Id_{H_c}, \ S_0= S_2S_1S_2, \ S_0S_1|_{H_c^1}= -Id_{H_c^1}= \hat S|_{H_c^1}.
\]
\item We have
\[
(\hat S \otimes \hat S) \Delta= \Delta^{op}.
\]
\item $M_1=H_c \otimes (N_0\otimes N_0^*)$ for some left $H_c$-module $(N_0, \mu_{N_0})$. $H_c$ acts on $N_0^*$ from the left using $\mu_{N_0}\circ S_0$. $N_0\otimes N_0^*$ carries the diagonal left $H_c$-action (\ref{diagact}), which we call $\mu_{N_0\otimes N_0^*}$. $M_1$ carries the diagonal left $H_c$-action $\mu_{M_1}$ given by $\mu_{N_0\otimes N_0^*}$ and the natural left action of $H_c$ on itself. $M_1$ is then a left Hopf module with coaction $\rho_{M_1}=\Delta \otimes Id$.
\item As a set, $M_0=H_c \otimes (\hat N_0\otimes \hat N_0^*)$ for some left $H_c$-module $(\hat N_0, \mu_{N_0})$. Let $H_c$ act on $\hat N_0$ from the left using $\mu_{\hat N_0}\circ (S_2)$. Let $H_c$ act on $\hat N_0^*$ from the left using $\mu_{\hat N_0}\circ (S_1S_2)$. $\hat N_0\otimes \hat N_0^*$ carries the diagonal left $H_c$-action (\ref{diagact}), which we call $\tilde \mu_{\hat N_0\otimes \hat N_0^*}$. $M_0$ carries the diagonal left $H_c$-action $\mu_{M_0}$ given by $\mu_{M_0\otimes M_0^*}$ and the natural left action of $H_c$ on itself. $M_0$ is then a left Hopf module with coaction $\rho_{M_0}=\Delta \otimes Id$.
\item Let $D_4$ be the dihedral group (symmetry group of the square), let $r_1$ represent a generator for the subgroup of rotations $R_4\subset D_4$, let $s_1$ and $s_2$ be two non-consecutive reflections in $D_4$. Then there is a homomorphism $\theta: D_4 \rightarrow {\rm Aut}(H_c^1)$ so that
\[
\theta(s_1)=S_0, \ \theta(s_2)=S_1, \ \theta(r_1)=S_2.
\]
In especially, the subgroup of ${\rm Aut}(H_c^1)$ generated by the image of $\theta$ can be viewed as a subquotient of the automorphism group of the quaternions $Q$, since there is a monomorphism $\tau:D_4\rightarrow S_4\simeq {\rm Aut}(Q)$.

\end{enumerate}
\begin{Def}\label{quaternionic}
A structure $(H^*_c, \omega, \mu, \eta, \Delta, \epsilon, \hat S, S_0, S_1, S_2, N_0, \hat N_0, M_0, M_1, \theta)$ as described above is called (commutative and cocommutative) {\it dihedral Lagrangian Hopf module}. 
\end{Def}

{\it Remark.} Note that the two 'opposite' Hopf algebras $\hat H$ and $\hat H^{op}$ are easily seen to be isomorphic using the isomorphism of graded Hopf algebras $\hat S:\hat H \rightarrow \hat H^{op}$, which simply comes down to $-Id$ restricted to $H_c^1$ (recall that $-Id_{H_c^1}= \hat S|_{H_c^1}$ by axiom 2.). In this sense, this natural isomorphism (antipode) $\hat S$ is factorized on $H^1$ by the requirement $S_0S_1|_{H_c^1}= -Id_{H_c^1}= \hat S|_{H_c^1}$ (note however that $S_0S_1$ and $\hat S$ define on $H^*_c$ distinct mappings, $\hat S$ being a $B$-antihomomorphism and $S_0S_1$ being a $B$-homomorphism). From this viewpoint, the Theorems \ref{hopf} and \ref{theorem2} seem to stand in a somewhat 'hidden' relation to the 'Theorem of the cube' (cf. Mumford \cite{mumford}) and related theorems in the theory of theta functions.

We then have the following almost immediate proposition:

\begin{prop}
$M=M_0\otimes M_1$ is a Hopf module over the Hopf algebra $\hat H$, where $M$ carries the diagonal left $H_c$-action $\mu_M$ induced by $(\Delta_0, \mu_{M_0}, \mu_{M_1})$ using (\ref{diagact}) and the codiagonal left action $\rho_M$ induced by $(\rho_{M_0}, \rho_{M_1}, \mu_{M_0}, \mu_{M_1})$ using (\ref{codiagact}).
\end{prop}
\begin{proof} Note there is something to prove here since $(M, \mu_{M_1}, \rho_{M_1})$ is a priori no Hopf module over $H_0$. But it follows immediately from coassociativity of $\Delta_0$ and associativity of $\mu_{M_0}, \mu_{M_1}$ that the associativity in (\ref{diagact}) holds. On the other hand, coassociativity of $\rho_M$ then follows using the formula (\ref{codiagact}), the fact that $\Delta_0$ and $\Delta_1$ are related by 3. in the above axioms and the fact that $S_0, S_1$ are graded algebra homomorphisms, we omit the calculations here to the reader.
\end{proof}
Before we proceed, we have to take a short digression to clarify the notation. We have defined in the proof of Proposition \ref{classi} resp. in Theorem \ref{genclass} a certain deformation of a given $\nabla$-parallel almost complex structure $J$ on $M\setminus \mathcal{C}$, where $\mathcal{C}$ is the critical set of a certain smooth closed section $s:M \rightarrow T^*M$. This deformation depends on $s$ and furthermore on a chosen $O(n)$-reduction of the symplectic frame bundle $R$ of $(M,\omega)$. The endpoint of this deformation in $\mathcal{J}(M\setminus \mathcal{C})$, together with the primordial closed section of $P_{G/\tilde G}$, defines a (rigid) singular Frobenius structure in the sense of Definition \ref{frobenius}. Note that in the general (non-Kaehler) case, closed sections of $P_{G/\tilde G}$, with the notation of Theorem \ref{genclass} will not define a regular Frobenius structure, since the spinor connection $\nabla$ induced by a given $U(n)$-reduction of $R$ associated to a given parallel almost complex structure $J$ does in general not preserve sections of the complex line bundle $\mathcal{E}_M$ described in the proof of Theorem \ref{genclass}. However, we might want to speak about such structures in distinction to the (rigid) singular structures alluded to above. 
\begin{Def}
A $5$-tuple $(\mathcal{L}, \mathcal{A}, \nabla, <\cdot,\cdot>, \mathcal{E})$ on a general symplectic manifold $(M,\omega)$ satisfying all axioms of Definition \ref{frobenius} except $\nabla \Gamma(\mathcal{L})\subset \Gamma(\mathcal{L})$ will be called a pre-Frobenius structure. If a pre-Frobenius structure $(\mathcal{L}, \mathcal{A}, \nabla, <\cdot,\cdot>, \mathcal{E})$ results from a closed section $s$ of $P_{G/\tilde G}$ in the sense of Proposition \ref{classi} and Theorem \ref{genclass} resp. the discussion above we will call it standard. In this situation, the (rigid) singular standard Frobenius structure induced by $s$ by the same proposition and theorem will be called the Frobenius structure associated to the implied standard pre-Frobenius structure.
\end{Def}
The following theorem is formulated for pre-Frobenius structures associated to closed sections $s$ of $P_{G/\tilde G}$ which are $C^1$-small and their associated (rigid) singular Frobenius structures. The formulation for the latter (the singular case) is preliminary here since a more meaningful Theorem can be proven when the appropriate notions of matrix factorization have been developed in the next section. For the same reason, we will ourselves limit here to the case of 'canonical' dual pairs in the sense of Theorem \ref{hopf}, while the more general dual pairs of Theorem \ref{theorem2} will be examined after the appropriate notion of matrix factorization in Section \ref{matrixfact2} is in place.
\begin{theorem}
To each (standard) dual pair of irreducible (in general weak) generalized standard pre-Frobenius structures $(\Omega, \mathcal{L})$ resp. $(\hat \Omega, \hat {\mathcal{L}})$ in the sense of Definition \ref{frobenius} and Theorem \ref{hopf} resp. to its singular counterpart we can associate a (commutative and cocommutative) {\it dihedral Lagrangian Hopf module} in the sense of Definition \ref{quaternionic} in a canonical (that is essentially unique) way.
\end{theorem}

\subsubsection{Matrix factorization}\label{matrixfact2}
Before we address the questions posed at the end of Section \ref{semic} we will still give some brief remarks on another aspect of Theorem \ref{theorem2}, commonly known as 'matrix factorization'. In the literature (\cite{eisenbud}, \cite{orlov} and references therein), the category of matrix factorizations is for instance understood as consisting of a differential $\Z/2\Z$-graded category $DB_{w_0}(W)$ that on the level of objects consists of ordered pairs $\overline P=(P_1, P_0)$ of (finitely generated, projective) $A$-modules on an affine scheme $X={\rm Spec}(A)$ and pairs of morphisms $\overline p=(p_1,p_0), p_1:P_1 \rightarrow P_0, p_0: P_0 \rightarrow P_1$ so that there is a distinguished point $w_0 \in X$ and a flat morphism $W:X \rightarrow \mathbb{A}^1$ satisfying
\begin{equation}\label{matrixfact}
p_1\circ p_0 = p_0\circ p_1 =W - w_0,
\end{equation}
where we consider $W,w_0 \in A$. Morphisms are given by the $\Z/2\Z$-graded complex $\mathbb{H}(\overline Q, \overline P)=\bigoplus_{i,j} Hom(Q_i,P_j)$ with grading $(i-j) {\rm mod}\ 2$ and differential $D$ acting on morphisms of degree $k$ as $Df=\overline q\circ f -(-1)^kf \circ \overline p$. Note that for a given pair $\overline P=(P_1, P_0)$ of $A$-modules and pairs of $A$-module maps $\overline p=(p_1,p_0), p_1:P_1 \rightarrow P_0, p_0: P_0 \rightarrow P_1$ and an element $x \in A$ so that $p_1\circ p_0=(x){\bf 1}_{P_0}$ and $p_0\circ p_1=(x){\bf 1}_{P_1}$ we can associate a $2$-periodic complex over the ring $B=A/(x)$
\begin{equation}\label{matrixf}
\mathcal{P}(p_1, p_0): \dots \rightarrow \overline P_1\rightarrow \overline P_0\rightarrow \overline P_1\rightarrow {\rm coker}(p_0)\rightarrow 0
\end{equation}
which is $B$-free and exact if $(x)/(x)^2$ is free over $B$ and thus a resolution of ${\rm coker}(p_0)$, where $\overline {(\cdot)}$ denotes reduction ${\rm mod} (x)$. For any $B$-module $E$, we define $H^*(\overline P, E)$ as the $\Z_2$-graded (co)homology group of the image of the complex $\mathcal{P}(p_1, p_0)$ under $Hom(\cdot, E)$.
\\
Consider now the maps $\hat s_l, \hat s_2: M\rightarrow \hat P_{G/\tilde G}\simeq P_{G/\tilde G}\simeq \hat P/\tilde G$ as introduced above Theorem \ref{theorem2} defining a pair of (singular) irreducible (generalized) standard Frobenius structures with associated closed sections $s_l, s_2: M\rightarrow T^*M$ using $P_G= (\pi_P^*(T^*M) \times_{Mp(2n, \R)} P)\simeq \hat P_G=P\times_{Mp(2n, \R), {\rm Ad}} G$ and a fixed $\tilde G$-reduction of $P$. As above, we can associate a map $\overline T:P\rightarrow Mp(2n, \R)/\hat U(n)$ so that $\pi_{Mp}(\hat s_l).\overline T(p)=\pi_{Mp}(\hat s_2),\ p\in P$ if we identify $\hat s_l, \hat s_2:P_G\rightarrow G/\tilde G$. On the other hand we denote by $\mathcal{T}_{l,2}$ the set of fibrewise linear symplectic (smooth) vector bundle automorphisms of $T^*M\setminus \pi^{-1}(\mathcal{C}_{l,2})$ covering the identity on $M$ (we can here consider the bundle $T^*M$ as a symplectic fibre bundle over $M$, using the identification by $\omega:T^*M \simeq TM$ and thus identifying the vertical bundle $V^*M\subset T(T^*M)$ with $TM$) and satisfying
\begin{equation}\label{linearmodule}
T_{l,2}\in {\rm End}_\omega(T^*M\setminus \pi^{-1}(\mathcal{C}_{l,2})),\ T_{l,2}(s_l(x)) =s_2(x),\ x \in \mathcal{C}_{l,2},
\end{equation}
where $\pi:T^*M \rightarrow M$ and $\mathcal{C}_{l,2}$ is the union of the (as we assume) transversal intersections of $s_l$ and $s_2$ with the zero-section of $T^*M$ (compare Definition \ref{frobenius}) which we assume to be a finite set of isolated points in $M$, in fact it follows that the $s_l, s_2$ are described locally by Morse type singularities. Then given an automorphism $g:P_G\rightarrow P_G$ covering the identity over $M$, that is a section of 
\begin{equation}\label{blaext}
P_G[G]=P_G \times_{G, Ad} G,\quad {\rm where} \ (p,g) \mapsto (pa, {\rm Ad}_a(g)), \ g,a \in G,
\end{equation}
satisfying $\hat s_l.g=\hat s_2$ and being represented by an $Ad$-equivariant map $\tilde g: P_G \rightarrow G$ via $j:(H_n\times_\rho Mp(2n, \R))/{\tilde G}\rightarrow (H_n\times_\rho Mp(2n, \R))/G_0$ to arrive at a map $\tilde T_{l,2}: P_G \rightarrow (H_n\times_\rho Mp(2n, \R))/G_0$. We recall that writing $P_G=\pi_P^*(TM) \times_{Mp(2n, \R)} P$ as defined in (\ref{tangent}) we denoted ${\rm pr}_1:P_G\rightarrow T^*M$ the map $pr_1((y,q), x), (p, x))=((gy,q),x),\ x \in M,\ y \in \R^{2n},\ p,q \in P_{\tilde G},\ q=p.g,\ g \in \tilde G$ (using the complex structure $J$ corresponding to $P_{\tilde G}$ to identify $TM\simeq T^*M$). We then claim
\begin{lemma}\label{affine}
For any $Ad$-equivariant map $\tilde g: P_G \rightarrow G$ satisfying $\hat s_l.\tilde g=\hat s_2$ for two given sections $\hat s_l, \hat s_2$ of $\hat P_{G/\tilde G}\simeq P_{G/\tilde G}\simeq \hat P/\tilde G$ and any section $\mathfrak{s}_1: T^*M \rightarrow P_G$ of ${\rm pr}_1:P_G\rightarrow T^*M$, that is ${\rm pr}_1\circ \mathfrak{s}_1 =Id$ there is a well-defined (symplectic affine-linear) automorphism $T_{l,2}\in {\rm End}_{\omega,{\rm aff}}(T^*M)$ of $T^*M$ covering the identity on $M$ so that $T_{l,2}={\rm pr}_1\circ g\circ \mathfrak{s}_1$. Outside of $\mathcal{C}_{l,2}$, $\tilde g$ can be chosen so that this automorphism is fibrewise symplectic linear, so $T_{l,2}\in \mathcal{T}_{l,2}$.
\end{lemma}
\begin{proof}
Of course the well definedness follows from the well-definedness of ${\rm pr}_1, g, \mathfrak{s}_1$. Outside of $\mathcal{C}_{l,2}$, $s_l$ and $s_2$ are given by sections of $P_G/\tilde G=(\pi_P^*(TM) \times_{Mp(2n, \R)} P)/\tilde G$ being locally represented by elements of the type $((y,q), x), (p, x), x \in M, 0 \neq y\in \R^{2n},\ q,p \in P$. Recall that $G$ acts on these elements from the right as in (\ref{tangentaction}), thus as
\begin{equation}\label{tangentaction2}
\tilde \mu: G\times P_G\rightarrow P_G, \ \tilde \mu\left((h, \hat g),(((y,q),p),x)\right)=\left(((\rho(g)^{-1}(y)+h, q.\hat g), p),x \right).
\end{equation}
Since $Sp(2n,\R)\subset G$ acts transitively on the set of vectors having the $\R^{2n}\setminus \{0\}$ (in fact $Sp(2n, \R)$ acts transitively on the set of $\omega_0$-symplectic bases and any non-zero vector in $\R^{2n}$ can be extended to a symplectic basis wrt $\omega_0$) thus on the set of elements of $P_G$ with local representatives $y \neq 0$ as above.
\end{proof}
Summarizing the above, we have associated to any pair of maps $\hat s_l, \hat s_2: M\rightarrow \hat P_{G/\tilde G}\simeq P_{G/\tilde G}\simeq \hat P/\tilde G$ with associated closed sections $s_l, s_2: M\rightarrow T^*M$ (using $P_G= (\pi_P^*(T^*M) \times_{Mp(2n, \R)} P)\simeq \hat P_G=P\times_{Mp(2n, \R), {\rm Ad}} G$) an equivariant map $\overline T:P\rightarrow Mp(2n, \R)/\hat U(n)$ as well as, fixing an $Ad$-equivariant map $\tilde g: P_G \rightarrow G$ satisfying $\hat s_l.\tilde g=\hat s_2$, an affine linear endomorphism $T_{l,2}\in {\rm End}_{\omega,{\rm aff}}(T^*M)$ of $T^*M$ covering the identity on $M$ so that outside of $\mathcal{C}_{l,2}$ we can assume $T_{l,2}$ to be fibrewise symplectic linear.\\
Consider now the decomposition $\mathcal{C}_{l,2}=\mathcal{C}_{l}\cup \mathcal{C}_{2}\subset M$ where $\mathcal{C}_{l}, \mathcal{C}_{2}$ are the sets of intersection of the above fixed closed sections $s_l, s_2: M\rightarrow T^*M$ with the zero section of $T^*M$, respectively, we assume $\mathcal{C}_{l}, \mathcal{C}_{2}$ to be isolated and finite (and $M$ to be compact or compact with boundary), furthermore we will assume in the following always that $\mathcal{C}_l\cap \mathcal{C}_2=\emptyset$. Consider that to any section $\hat s: M\rightarrow P_G$ associated to a closed section $s: M \rightarrow T^*M$ which intersects $M \subset T^*M$ transversally on $\mathcal{C}_s$ and an $\hat U(n)$-reduction of $P$ we can associate the $C^\infty(M \setminus \mathcal{C}_{s})$-module of smooth sections $\mathcal{E}(\mathcal{L}_s, M\setminus \mathcal{C}_{s})=\Gamma(\mathcal{L}_s)|M\setminus \mathcal{C}_{s})$ of the associated standard irreducible, in general singular (cf. Proposition \ref{classi}) Frobenius structure $(\Omega, \mathcal{L}_s)$ over the open set $M \setminus \mathcal{C}_{s}$ and also $\mathcal{E}(T^*M, M\setminus \mathcal{C}_{s})=\{s \in\Gamma(T^*(M\setminus \mathcal{C}_{s}))\}$, the $C^\infty(M \setminus \mathcal{C}_{s})$-module of (possibly singular) sections of $T^*M$ which are smooth over the open set $M\setminus \mathcal{C}_{s}$. Now let $T \in {\rm End}_\omega(T^*M\setminus \pi^{-1}(\mathcal{C}_{l,2}))$ be a linear automorphism of $T^*M\setminus \pi^{-1}(\mathcal{C}_{l,2})$ covering the identity and satisfying (\ref{linearmodule}), it is clear that $T$ defines an endomorphism of $\mathcal{E}(T^*M, M\setminus \mathcal{C}_{l})$ into $\mathcal{E}(T^*M, M\setminus \mathcal{C}_{2})$ induced by continuation by zero through $\mathcal{C}_2$. Let now for a multiindex $r=(r_1, \dots r_k), k=|\mathcal{C}_{l}|, r_i \in \N^+$ be $\mathcal{M}^r(T^*M, \mathcal{C}_{l}), r_i\geq 0, r_i \in \N$ the submodule of $\mathcal{E}(T^*M, M)$ generated by the $r$-th power of the ideal $\mathfrak{p}_l=C^\infty(M ,\mathcal{C}_{l})$ of smooth functions in $C^\infty(M, \C)$ on $M$ that vanish on $\mathcal{C}_{l}$, that is if $\mathfrak{p}_l=\Pi_{i=1}^k\mathfrak{m}_{x_i}, x_i \in \mathcal{C}_l$, $\mathfrak{m}_{x_i}$ maximal ideals at the $x_i \in \mathcal{C}_l$, we set
\[
\mathfrak{p}^r_l=\Pi_{i=1}^k\mathfrak{m}^{r_i}_{x_i}
\]
and consider $\mathcal{M}^r(T^*M, \mathcal{C}_{l})$ as a $\mathfrak{p}_l$-submodule of $\mathcal{E}(T^*M, M\setminus \mathcal{C}_{l})$ resp. a $C^\infty(M, \C)$-module of $\mathcal{E}(T^*M, M\setminus \mathcal{C}_{l})$ by restriction, analogously for the pair $(\hat s_2, \mathcal{C}_2)$. Then any element $T \in {\rm End}_\omega(T^*M\setminus \pi^{-1}(\mathcal{C}_{l}))$ whose singularity at $\mathcal{C}_l$ is annihilated by some element $\mathfrak{c}$ of $\mathfrak{p}^r_l$ satisfies $T(\mathfrak{c}\mathcal{M}^0(T^*M, \mathcal{C}_{l}))\subset  \mathcal{E}(T^*M, M\setminus \mathcal{C}_{2})$, where $r$ is the pole order of $T$, to be defined below and $\mathfrak{c}^{-1}$ is regarded as an element of the quotient field of $R=C^\infty(M,\C)$. We will denote the subset of ${\rm End}_\omega(T^*M\setminus \pi^{-1}(\mathcal{C}_{l}))$ whose 'singularities at $\mathcal{C}_{l}$ are annihilated' near $\mathcal{C}_l$ by appropriate elements of $\mathfrak{p}_l^r$ by ${\rm End}^r_\omega(T^*M\setminus \pi^{-1}(\mathcal{C}_{l}))$ (analogously for $\mathcal{C}_2$). Note that 'being annihilated near $\mathcal{C}_l$' means here and in the above that there exists a $g \in \mathfrak{p}_l^r$ resp. for any $x_i \in \mathcal{C}_l$ an element $g_{x_i}$ of the power $\mathfrak{m}^{r_i}_{x_i}$ of the maximal ideal of smooth functions at $x_i \in \mathcal{C}_l$ so that 
\begin{equation}\label{maximalideal}
(g_{x_i}T)(s) \in \mathcal{E}(T^*M, M),\ {\rm for\ all}\ s \in \mathcal{E}(T^*M, M), \ x_i \in \mathcal{C}_l, 
\end{equation}
where $T \in {\rm End}_\omega(T^*M\setminus \pi^{-1}(\mathcal{C}_{l}))$ and no element $g_{x_i}\in \mathfrak{m}^{r}_{x_i}$ with non-trivial image in $\mathfrak{m}^{r}_{x_i}/ \mathfrak{m}^{r_i}_{x_i}$ and $r< r_i$ will satisfy the above. To understand the prior condition better we can reformulate it in terms of 'formal Laurent expansions' using Whitney's theorem \cite{whitney}, that is we have:
\begin{lemma}\label{laurent1}
Let $T \in {\rm End}_\omega(T^*M\setminus \pi^{-1}(\mathcal{C}_{l}))$, then $T \in {\rm End}^r_\omega(T^*M\setminus \pi^{-1}(\mathcal{C}_{l}))$ if and only if for any $x_i \in \mathcal{C}_l$, there exist open sets $x_i\in U_{x_i}\subset M$ and diffeomorphisms $\Psi_{x_i}:(U_{x_i},x_i)\rightarrow (\R^n,0)$ so that if $\tilde T_{jk}\in C^\infty(\R^n\setminus\{0\})$ is any matrix entry of $\tilde T=D_{x_i}\Psi\circ T\circ D_0\Psi^{-1} \in End(T^*(\R^n\setminus \{0\}))$ we have 
\begin{equation}\label{laurent}
{\rm min}_{j,k}\{m\in \N: J_0^m(x^{r_i}\tilde T_{jk})\neq {\bf 0}\}=0.
\end{equation}
where $J_0^m: C^\infty(\R^n)\rightarrow T^m_0$ is the $m$-th degree Taylor polynomial of a smooth function on $\R^n$ at $0$ and $x^{r_i}$ is a homogeneous polynomial of degree $r_i$.
\end{lemma} 
\begin{proof}
We can assume that the preimages of the coordinate functions $(y_1, \dots, y_n)$ under the diffeomorphism $\Psi_{x_i}:(U_{x_i},x_i)\rightarrow (\R^n,0)$ actually generate the maximal ideal $\mathfrak{m}_{x_i}$ at $x_i \in \mathcal{C}_l$. This is because we can choose Riemannian normal coordinates induced by the pair $(J, \omega)$ on the appropriate nghbhds $U_{x_i}$ and the distance function on $U_{x_i}$, restricted to the coordinate axes of $\Psi_{x_i}^{-1}(\R^n)$, thus translates into the linear coordinate functions on $\Psi_{x_i}(U_{x_i})$. Then assuming (\ref{maximalideal}) the equation (\ref{laurent}) follows immediately since $(\Psi_{x_i}^{-1})^*((g_{x_i}T)(s))_i \in C^\infty(\R^n), i=1, \dots, n$ for $r=r_i$ if $g_{x_i} \in \mathfrak{m}^{r}_{x_i}$ and no $g_{x_i}\in \mathfrak{m}^{r}_{x_i}$ with non-empty image in $\mathfrak{m}^{r}_{x_i}/ \mathfrak{m}^{r_i}_{x_i}$ and $r< r_i$ will achieve smoothness. On the other hand, assuming (\ref{laurent}), taking the same chart $\Psi_x$ as above and constructing by Whitney's Theorem a (matrix) function $h_{jk} \in  C^m(\R^n)\ , m>0$ (in this case simply the Taylor polynomial to degree $m$) whose term-wise Taylor polynomial at zero coincides with $J_0^m(x^{r_i}\tilde T_{jk})$ for any fixed $m>0$ we see that $(\Psi_{x_i}^{*})(x^{r_i}h.\Psi_{x_i}(s))$ on $U_{x_i}$ sufficiently small actually coincides with $(g_{x_i}T)(s)|U_{x_i}$ (up to arbitrarily high order $m$) for a given $s \in \mathcal{E}(T^*M, M)$ and thus we also have $(g_{x_i}T)(s)\in \mathcal{E}(T^*M, M)$.
\end{proof}
In other words, the entries of $T$ at each $x_i \in \mathcal{C}_l$ have formal Laurent expansions with maximal pole order $r_i$. Let for a given fixed $T \in {\rm End}^r_\omega(T^*M\setminus \pi^{-1}(\mathcal{C}_{l}))$ denote from now on $\mathcal{M}^{r, T}(T^*M, \mathcal{C}_{l})\subset \mathcal{M}^r(T^*M, \mathcal{C}_{l})$ the set of elements in $\mathcal{M}^r(T^*M, \mathcal{C}_{l})$ that actually satisfy an equation similar to (\ref{maximalideal}), that is $s \in \mathcal{M}^{r, T}(T^*M, \mathcal{C}_{l}) \subset \mathcal{M}^r(T^*M, \mathcal{C}_{l})$ if and only if
\begin{equation}\label{maximalideal3}
T(s) \in \mathcal{E}(T^*M, M),  
\end{equation}
it is a $C^\infty(M, \C)$-module. Note that the existence of a finite $r \in \N^k_+$ so that a given ${\rm End}_\omega(T^*M\setminus \pi^{-1}(\mathcal{C}_{l}))$ is in ${\rm End}^r_\omega(T^*M\setminus \pi^{-1}(\mathcal{C}_{l}))$ has to be assumed, that is in the following we will assume that the following is valid:
\begin{asslemma}
Let $\hat s_l, \hat s_2: M\rightarrow \hat P_{G/\tilde G}$ wrt a given $\tilde G \subset \hat U(n)$-reduction of $P$ be given so that for the associated closed sections $s_l, s_2: M\rightarrow T^*M$ and  $T_{l,2} \in {\rm End}_\omega(T^*M\setminus \pi^{-1}(\mathcal{C}_{l}))$ satisfying (\ref{linearmodule}) there exists a finite $r \in \N_+^k$ so that in fact $T_{l,2} \in {\rm End}^r_\omega(T^*M\setminus \pi^{-1}(\mathcal{C}_{l}))$, which is the case if and only if there exist for any $x \in M$ a sufficiently small nghbhd $U_x\subset M$ and Riemannian normal coordinates on $U_x$ wrt to the chosen tupel $(\omega, J)$ so that with $r=(r_1, \dots, r_k) \in N_+^k$ one has that (\ref{laurent}) is satisfied.
\end{asslemma}
Of course the proof of this last assertion follows immediately from Lemma \ref{laurent1}. We then get for $r=(r_1, \dots r_k),\ r_i>0$ by continuous extension a map $\Psi: {\rm End}^r_\omega(T^*M\setminus \pi^{-1}(\mathcal{C}_{l}))\rightarrow \mathcal{P}(\mathcal{M}^r(T^*M, \mathcal{C}_{l}))$, where $\mathcal{P}(M)$ denotes the power set of a set $M$ given by $\Psi(T)=\mathcal{M}^{r, T}(T^*M, \mathcal{C}_{l})$ and if ${\rm ev}:\mathcal{P}(\mathcal{M}^r(T^*M, \mathcal{C}_{l}))\rightarrow \mathcal{M}^r(T^*M, \mathcal{C}_{l})$ is the map that assigns to each subset the set of its elements we arrive at a well-defined map
\[
\mathcal{T}: {\rm graph}(({\rm Id},{\rm ev})\circ \Psi)\subset {\rm End}^r_\omega(T^*M\setminus \pi^{-1}(\mathcal{C}_{l}))\times \mathcal{M}^r(T^*M, \mathcal{C}_{l})\rightarrow \mathcal{E}(T^*M, M\setminus \mathcal{C}_{2}),
\]
where $\mathcal{C}_{2}$ is arbitrary and $\mathcal{M}^r(T^*M, \mathcal{C}_{l})$ is generated as described priorly by $\mathfrak{p}^r_l= \Pi_{i=1}^k\mathfrak{m}^{r_i}_{x_i}$ in $\mathcal{M}(T^*M, \mathcal{C}_{l}):= \mathcal{M}^{\bf 1}(T^*M, \mathcal{C}_{l})$ for some appropriate $r=(r_1, \dots, r_k)$ and ${\bf 1}=(1, \dots, 1)$. It is then clear that for a fixed $T \in {\rm End}^{\bf 1}_\omega(T^*M\setminus \pi^{-1}(\mathcal{C}_{l})$, the implied endomorphism $T:\mathcal{M}^T(T^*M, \mathcal{C}_{l})=\mathcal{M}^{{\bf 1},T}(T^*M, \mathcal{C}_{l})\rightarrow \mathcal{E}(T^*M, M\setminus \mathcal{C}_{2})$ has cokernel isomorphic to $\mathcal{E}(T^*M, M\setminus \mathcal{C}_{2})/\mathcal{M}^T(T^*M, \mathcal{C}_{l})\simeq \C^k$ if we consider $\mathcal{M}^T(T^*M, \mathcal{C}_{l})\subset \mathcal{E}(T^*M, M\setminus \mathcal{C}_{2})$ and assume $\mathcal{C}_l\cap \mathcal{C}_k=\emptyset$. Consider now a map $T_{l,2}\in {\rm End}_{\omega,{\rm aff}}(T^*M)$ associated to the sections $\hat s_l, \hat s_2: M\rightarrow \hat P_{G/\tilde G}$ as described above, we will see below (Lemma \ref{ass3}) that we can consider it as a fibrewise linear map $T_{l,2} \in {\rm End}^r_\omega(T^*M\setminus \pi^{-1}(\mathcal{C}_{l}))$ for some appropriate $r$ as above (see also Assumption/Lemma \ref{ass3} below) and thus as an endomorphism $T_{l,2}:\mathcal{M}^{r, T}(T^*M, \mathcal{C}_{l})\rightarrow \mathcal{E}(T^*M, M\setminus \mathcal{C}_{2})$. On the other hand, if $\mathcal{J}$ is the set of $\omega$-compatible almost complex structures on $M$, we can view $\overline T:P\rightarrow Mp(2n, \R)/\hat U(n)$ a priorily as a map $\overline T \in {\rm End}(\mathcal{J})$ that maps by definition $J$ to $J_T$ by using the pointwise identification $\mathcal{J}_x \simeq \mathfrak{h}$ and the corresponding action (by conjugation) of $Sp(2n, \R)$ on $\mathfrak{h}$ with stabilizer $U(n)$. Given $\hat s_l, \hat s_2:P_G\rightarrow G/\tilde G$ and assuming (as we will do throughout in the following) that $P^J_L$ and $P^{J_T}_L$ are isomorphic as $\hat O(n)$-reductions of $P$ we have a global section $\varphi$ of $(P^J_L)_{\hat O(n)}(Sp(2n, \R))=P^J_L\times_{\hat O(n)} Sp(2n, \R)$, where here $\hat O(n)$ acts by the adjoint mapping so that the diagram 
\begin{equation}\label{dia1}
\begin{diagram}
(P^J_L)_{\hat O(n)}(Mp(2n, \R))   &\rTo^{\sigma} &(P^J_L)_{\hat O(n)}(Mp(2n, \R)/\hat O(n))  \\
\uTo_{\varphi}     &\ruTo_{\pi_{Mp}\circ \hat s_2 }      \\
M              &     
\end{diagram}
\end{equation}
where $(P^J_L)_{\hat O(n)}(Sp(2n, \R)/\hat O(n))=P^J_L\times_{\hat O(n)} Sp(2n, \R)/\hat O(n)$ (again with $\hat O(n)$ acting via $Ad$), $\sigma$ is the natural functorialism on associated bundles induced by the canonical projection $Sp(2n, \R)\rightarrow Sp(2n, \R)/\hat O(n)$ and $\pi_{Mp}\circ \hat s_2:M\rightarrow (P^J_L)_{\hat O(n)}(Mp(2n, \R)/\hat O(n))$ defining the $\hat O(n)$ reduction of $P$ induced by $J_{T}$ and $L$. We can thus compare $\pi_{Mp}(\hat s_l), \pi_{Mp}(\hat s_2)$ using the $\hat O(n)$-identity section representative in $P/\hat O(n)=P\times Mp(2n, \R)/\hat O(n)$ to arrive at an ($\hat O(n)$-equivariant) map $\overline T=\varphi: P^J_L\rightarrow Sp(2n, \R)$. Alternatively, if ${\rm Gr}(T^\mathcal{J}_{\C}M^\pm) =\{{\rm Im}(\alpha^\pm_J(TM)), \ J \in \mathcal{J}\}$ is the set of the $\pm i$-eigenspaces parametrized by the set $\{J \in \mathcal{J}\}$ in $T_{\C}M$, we can view $\overline T$ as a map on the set ${\rm Gr}(T^\mathcal{J}_{\C}M^\pm)$ (covering the identity on $M$) which sends $T^J_{\C}M^\pm:={\rm Im}(\alpha^\pm_J(TM))$ to $T^{J_T}_{\C}M^\pm:={\rm Im}(\alpha^\pm_{J_T}(TM))$ resp. the associated projection $\alpha^\pm_J$ to $\alpha^\pm_{J_T}$, depending on perspective, so that we get also a mapping on $\mathcal{P}^{\mathcal{J}, \pm}$, the latter being the set of projections on $T^*M$ onto the set ${\rm Gr}(T^\mathcal{J}_{\C}M^\pm)$. In fact we can relate this action with the natural fibrewise linear action of any equivariant section $\overline T=\varphi: P^J_L\rightarrow Sp(2n, \R)$ on $T_{C}^M$. For this consider that we can parameterize the elements $L\in \mathfrak{h}$, the Siegel upper half-space (compare Section \ref{coherent}) of symplectic standard space $(V=\R^{2n}, \omega_0, J_0)$, resp. the set of totally complex positive Lagrangian subspaces of $V_\C$ resp the set of symmetric, positive, antilinear maps $T(L): V\rightarrow V$ satisfying
\[
L=\{\alpha^+_{J_0}(x)+ \alpha^-_{J_0}(T(L)x): x \in V\}\subset V_\C,
\]
where symmetry of $T(L)$ is measured wrt the real part of the Hermitian form $\langle x,y\rangle_0= \omega_0(\cdot, J_0 \cdot)+ i\omega_0(\cdot, J_0 \cdot)$, 'positivity' of $T(L)$ means here $\langle T(L)x,T(L)y\rangle_0 < \langle x,y\rangle_0 \ \forall x,y \in V$ and positivity of $L$ is measured wrt the Hermitian form $H(z,w)=\omega_0(z, \overline w), z,w  \in V_\C$, for more details cf. Sternberg \cite{sternberg}. We then have, by following loc. cit., that when considering the map $\mathcal{L}(T(L)): V\rightarrow V_C, \mathcal{L}(x)=\alpha^+_{J_0}(x)+ \alpha^-_{J_0}(T(L)x)$ that 
\[
(g\circ \mathcal{L}(T(L)))^{-1} \circ \mathcal{L}(T(gL))(x)= (BT(L)+A)^{-1}x,\ x  \in V,
\]
where $g=\left(\begin{smallmatrix} A&B\\B&A\end{smallmatrix}\right), A, B \in M(n, \R)$ is the decomposition of $g\in Sp(2n,\R)$ into $J_0$-linear and $J_0$-antilinear parts. Denoting $B(g, L)=BT(L)+A$, we note that by loc. cit., Chapter 5, $B(g, L)$ satisfies the coycle condition
\[
B(g_1g_2,L)=B(g_1,g_2L)B(g_2,L), \ g_1, g_2 \in Sp(2n,\R), L \in \mathfrak{h}.
\]
We will in the following understand $\overline T=\varphi: P^J_L\rightarrow Sp(2n, \R)$ as being the (essentially unique) fibrewise symplectic map on $(T_{\C}M, \omega)$ that is induced by the map $\overline T \in {\rm End}(\mathcal{J})$ interpreted as a map $g:P\rightarrow Mp(2n, \R)/\hat U(n))$ constructed above under the above isomorphism of symplectic actions on $\mathfrak{h}$ resp. on the set of (fibrewise) positive Lagangians of $T_{\C}M$, intertwining the implied action on $V_\C$ (using $\mathcal{L}$ up to the fibrewise cocycle $B$ relative to the fixed element $J \in \mathcal{J}$). Note that the presence of $B$ means that $\overline T$ is symplectic as a vector bundle isomorphism on $(T_{\C}M, \omega)$ covering the identity, that is an element of ${\rm End}^r_\omega(T_{\C}^*M\setminus \pi^{-1}(\mathcal{C}_{l,2}))$, where the multiindex $r$ has to be specified, but the assignment $(g:P\rightarrow Mp(2n, \R)/\hat U(n)) \mapsto \overline T$ in this sense is not a homomorphism wrt the implicit $Sp(2n, \R)$-actions on $\mathfrak{h}$ resp. $V_{\C}$, due to the 'cocycle property' of $B$. Summarizing the above discussion, we associate to two elements $J, J_T \in \mathcal{J}$ (resp. a symplectic connection $\nabla$ so that $\nabla J=0$, an $\hat O(n)$-reduction o $P$ associated to $(\nabla, J)$ as above) and a chosen lift $\overline T=\varphi: P^J_L\rightarrow Sp(2n, \R)$ intertwining $J$ and $J_T$ in the above sense a smooth function $\mathcal{B}(J, J_T) \in {\rm End}_\omega(T_{\C}^*M)$ which is defined by
\begin{equation}\label{cocycle2}
\mathcal{B}(J, J_T)_x= B(g(x), L(x)), \ x \in M,
\end{equation}
wrt to local coordinates $L(x)$ being associated to $J(x)$ (in a local frame), $g(x)$ induced by $\overline T(x)$ as in the above, it follows immediately from the above that $\mathcal{B}(J, J_T)$ is independent of the local choice of $\hat O(n)$-frame. Note that for a given section $s\in \mathcal{E}(T^*M)$ and an element $J \in \mathcal{J}$ we can consider the onedimensional over $\C$ distribution $E_{J, s}=span_{\C}(s, J^*s)\subset T^*_{\C}M$ and consider the restriction $\mathcal{B}(J, J_T)|E_{J, s}$ (it follows from the above that $\mathcal{B}(J, J_T):E_{J, s}\rightarrow E_{J_T, s}$. Since $(s, J)$ resp. $(s, J_T)$ defines a canonical basis of $E_{J, s}$ resp. $E_{J_T, s}$ we can consider $\mathcal{B}(J, J_T)|E_{J, s}$ pointwise over $M$ as a conformal mapping (in fact, a Moebus transformation) on $\C$, we will write the corresponding complex-valued function as $\mathcal{B}_{\C}(J, J_T,s)$.
\\
We note (compare \cite{sternberg}) that any fixed smooth section $J: M\rightarrow P\times_{\hat U(n)} \mathcal{J}_0$ (here, $\mathcal{J}_0$ is the set of complex structures on symplectic standard space) in this sense gives rise to a cocyle 
\[
\chi_J: P\times_{\hat U(n), Ad} \mathcal{J} \rightarrow \R/{2\pi\Z}, \ (p, gT_0) \mapsto {\rm log}\left(\frac{det_{\C}{BT_0+A}}{|det_{\C}{BT_0+A}|}\right),
\]
where $P\times_{\hat U(n), Ad} \mathcal{J}_0=(P^J_L)_{\hat U(n), Ad}(Sp(2n, \R)/\hat U(n))=P^J_L\times_{\hat U(n), Ad} Sp(2n, \R)/\hat U(n)$ and we understand $J(x)=(p(x), T_0(x)), x \in M, p(x) \in (P_{\hat U(n)})_x$ and $g=\left(\begin{smallmatrix} A&B\\B&A\end{smallmatrix}\right), A, B \in M(n, \R)$ is the decomposition of $g\in Sp(2n,\R)$ into $J_0$-linear and $J_0$-antilinear parts, parameterizing the respective fibre $\mathcal{J}_x,\ x\in M$ as above and $log(re^{it})=ln(r)+it, r>0,\ t\in [0, 2\pi)$. The map $\chi_J$ is a coycle in the sense that if $g_1, g_2 \in {\rm End}_{\omega}(T^*M)$ are two fibrewise symplectic endomorphisms of $T^*M$ acting on $\Gamma(\mathcal{J})$ in the obvious way, then
\[
\chi_J(g_1g_2)=\chi_{g_2J}(g_1)+ \chi_{J}(g_2),
\]
for a proof see \cite{sternberg}. The cocycle $\chi_J$ will be needed particularly later on in the treatment of the non-$C^1$-small case using the symplectic spinor viewpoint. Note that $\chi_J$ descends to a $\R/\Z$-valued function on $M$ (exactly by pulling back by a secondary element $J_T \in \mathcal{J}$), thus defining (by \cite{cheegersimons}) an element $\tilde \chi_{J, J_T}$ of $H^1(M, \Z)$.
\begin{Def}
We will denote the element $\tilde \chi_{J,J_T}=J_T^*(\chi_J)$ of $H^1(M, \Z)$ associated to a given $\hat U(n)$ reduction of $P$ and given element $J, J_T\in \mathcal{J}$ the {\it relative symplectic genus} of $(M, \omega, J, J_T)$
\end{Def}
With the above conventions, of course we can consider the corresponding map on global sections $\overline T:\Gamma({\rm Gr}(T^\mathcal{J}_{\C}M^\pm))\rightarrow \Gamma({\rm Gr}(T^\mathcal{J}_{\C}M^\pm))$, we will denote $\Gamma({\rm Gr}(T^\mathcal{J}_{\C}M^\pm))$ as $\mathcal{E}({\rm Gr}(T^\mathcal{J}_{\C}M^\pm))$ in the following. Summarizing the above we have two maps associated to our pair of maps $\hat s_l, \hat s_2: M\rightarrow \hat P_{G/\tilde G}$ and fixing an $Ad$-equivariant map $\tilde g: P_G \rightarrow G$ satisfying $\hat s_l.\tilde g=\hat s_2$, namely for an appropriate $r\in \N^+$ 
\[
T_{l,2}:\mathcal{M}^{r, T_{l,2}}(T^*M, \mathcal{C}_{l})\rightarrow \mathcal{E}(T^*M, M\setminus \mathcal{C}_{2}), \quad \overline T: \mathcal{E}({\rm Gr}(T^\mathcal{J}_{\C}M^\pm)) \rightarrow \mathcal{E}({\rm Gr}(T^\mathcal{J}_{\C}M^\pm)).
\]
To see the symmetry in the above more clearly, note first that, after choosing local symplectic frames, the action of $Sp(2n, \R)$ on $\mathfrak{h}\simeq \mathcal{J}_x$ introduced in Section \ref{coherent} corresponds to the natural matrix action of ${\rm End}_\omega(T_{\C}M)$, thus equivariant maps $S:P\rightarrow Sp(2n, \R)$ on the Grassmannian $(T^\mathcal{J}_{\C}M^\pm)$, whose stabilisator is isomorphic to the set of identity maps in $S:P\rightarrow Sp(2n, \R)/\hat U(n)$. More precisely for a map $S:P\rightarrow Sp(2n, \R)/\hat U(n)$ we have $S.\alpha^\pm_J(TM)=\alpha^\pm_{S.J}(TM), \ J \in \mathcal{J}$ (as explained above we can choose a smooth representative $\overline T:P\rightarrow Sp(2n, \R)$). Here, we understand ${\rm Gr}(T^\mathcal{J}_{\C}M^\pm)$ as the fibre bundle ${\rm Gr}(T^\mathcal{J}_{\C}M^\pm)=P \times_{Mp(2n,\R)} {\rm Gr}(T^{\mathcal{J}_0}_{\C}\R^{2n})^\pm$, where $\mathcal{J}_0$ is the set of compatible complex structures on $(\R^{2n}, \omega_0)$. The above global action is then induced by the action of $Sp(2n, \R)$ on ${\rm Gr}((T^{\mathcal{J}_0}_{\C}\R^{2n})^\pm)$. Interpreting thus $\overline T: \mathcal{E}({\rm Gr}(T^\mathcal{J}_{\C}M^\pm)) \rightarrow \mathcal{E}({\rm Gr}(T^\mathcal{J}_{\C}M^\pm))$ as a map (denoted by the same symbol) $\overline T: \mathcal{E}(T_{\C}M) \rightarrow \mathcal{E}(T_{\C}M)$ and extending $T_{l,2}$ complex-linearly to a map $T_{l,2}:\mathcal{M}(T^*_{\C}M, \mathcal{C}_{l})\rightarrow \mathcal{E}(T^*_{\C}M, M\setminus \mathcal{C}_{2})$ and finally dualizing $\overline T$ using $\omega$ and extending the dualized map $\overline T^*: \mathcal{E}(T^*_{\C}M) \rightarrow \mathcal{E}(T^*_{\C}M)$ linearly to $\mathcal{E}(T^*_{\C}M, M \setminus \mathcal{C}_{2})$ we arrive at a pair
\begin{equation}\label{matrix1}
T_{l,2}:\mathcal{M}^{r,T_{l,2}}(T^*_{\C}M, \mathcal{C}_{l})\rightarrow \mathcal{E}(T_{\C}^*M, M\setminus \mathcal{C}_{2}), \quad \overline T^*: \mathcal{E}(T^*_{\C}M, M \setminus \mathcal{C}_{2}) \rightarrow \mathcal{E}(T^*_{\C}M, M \setminus \mathcal{C}_{l}).
\end{equation}
Note that on the level of elements of $T^*M$, $T^*_{\C}M$ resp. symplectic vector bundle bundle automorphisms covering the identity in $T^*M$ resp. $T^*_{\C}M$ (outside of $\mathcal{C}_{l,2}$) the above two mappings (before $\C$-linearly continuing $T_{l,2}$ to $T^*_{\C}M$, we will use the two views, $T_{l,2}$ as a real or complex-linear mapping interchangingly in the following) can be understood using the embeddings ($\R$-linear injections) $i\circ \alpha^\pm_{J}: TM \rightarrow T^{J}_{\C}M^\pm={\rm Im}(\alpha^\pm_{J}(TM))\hookrightarrow T_{\C}M$ resp. analogously $\alpha^\pm_{J_T}:TM\rightarrow T^{J_T}_{\C}M^\pm$ as being situated in ($\omega$-dual version of) the following 'commutative' diagram (in fact evidently only the 'big square' and the middle square commutes, not the 'small squares' on the right resp. the left)
\begin{equation}\label{dia2}
\begin{CD}
T(M\setminus \mathcal{C}_{l,2})  @>{\alpha^\pm_{J}}>> T_{\C}(M\setminus \mathcal{C}_{l,2})  @>{\overline T}>>  T_{\C}(M\setminus \mathcal{C}_{l,2})  @<{\alpha^\pm_{J_T}}<< T(M\setminus \mathcal{C}_{l,2})\\
@VV{T_{l,2}}V      @AA{i}A          @AA{i}A     @VV{T_{l,2}}V\\
T(M\setminus \mathcal{C}_{l,2})  @>{\alpha^\pm_{J}}>>  T^{J}_{\C}(M\setminus \mathcal{C}_{l,2})^\pm  @>{\overline T}>> T^{J_T}_{\C}(M\setminus \mathcal{C}_{l,2})^\pm @>{(\alpha^\pm_{J_T})^{-1}}>>T(M\setminus \mathcal{C}_{l,2})
\end{CD}
\end{equation}
we will prove the commutativity (in the above sense) of the diagram on appropriate spaces of (singular) sections of $T^*M$ resp. $T^*_{\C}M$ and respective automorphism spaces in Theorem \ref{theorem3}.\\
Assume now that the maps $\hat s_l, \hat s_2: M\rightarrow \hat P_{G/\tilde G}$ define a pair of {\it dual} (generalized standard irreducible, in general singular) Frobenius structures in the sense of Theorem \ref{theorem2} with $\mathcal{C}_{l}, \mathcal{C}_{2}$ the sets of (transversal) intersection of the associated closed sections $s_l, s_2: M\rightarrow T^*M$ with the zero-section of $T^*M$ and $T_{l,2}$ and $\overline T^*$ are associated to $\hat s_l, \hat s_2$ in the sense described above. We now want to argue how, by slightly enlarging $\mathcal{M}(T^*_{\C}M, \mathcal{C}_{l})$ while restricting $\mathcal{E}(T^*_{\C}M, M \setminus \mathcal{C}_{2})$ resp. $\mathcal{E}(T^*_{\C}M)$ we can modify the mappings in (\ref{matrix1}) into a pair of morphisms in the sense of (\ref{matrixfact}).\\
Let $\mathcal{E}^0(T^*M, M\setminus \mathcal{C}_{2})$ be the submodule of $\mathcal{E}(T^*M, M\setminus \mathcal{C}_{2})$ whose elements $s$ have poles on $\mathcal{C}_2$ that are actually annihilated near $\mathcal{C}_2$ (in the same sense as in (\ref{maximalideal3})) by an appropriate finite power of the ideal $C^\infty(M ,\mathcal{C}_{2})$ in $C^\infty(M, C)$, that is for any $s \in \mathcal{E}^0(T^*M, M\setminus \mathcal{C}_{2})$ there exists a multi-index $r\geq 0$ and a $g \in \mathfrak{p}_2^r$ so that $gs \in \mathcal{E}(T^*M, M)$. Define as $\mathcal{M}_\infty(T^*M, \mathcal{C}_{2})$ the submodule of $\mathcal{E}^0(T^*M, M\setminus \mathcal{C}_{2})$ generated by the ideal $C^\infty(M,\mathcal{C}_{2}, \PI^1)$ of smooth $\PI^1$-valued functions $C^\infty(M, \PI^1)$, so elements of the quotient field of $C^\infty(M, \C)$ on $M$, that either vanish or take the value $\infty\in \PI^1$ on $\mathcal{C}_{2}$. Here we identify $\PI^1\simeq S^2$, $S^2$ being the one point compactification of $\C$ and smoothness of a function $f$ near a pole $x \in \mathcal{C}_2$ is of course defined as smoothness of $1/f$ near $x$. In the following, we will restrict the maximal order of the poles on any $x_i \in \mathcal{C}_2$ in $\mathcal{M}_\infty(T^*M, \mathcal{C}_{2})$ resp. $C^\infty(M,\mathcal{C}_{2}, \PI^1)$ in the sense of (\ref{maximalideal}), denoting the resulting $C^\infty(M, \C)$-module resp. ring by $\mathcal{M}^0_\infty(T^*M, \mathcal{C}_{2})$ resp. $C^\infty_0(M,\mathcal{C}_{2}, \PI^1)$. The latter is endowed with a certain $Z_2$-graded ring structure if we assume it can be written as a principal fractional ideal in $K$, the quotient field of $C^\infty(M ,\C)$ with denominator having zeroes at most in $\mathcal{C}_{2}$, as follows. Let $v \in C^\infty(M, \C)$ (with zeroes in $\mathcal{C}_2$) and denote by $1/\mathfrak{v}$ the principal fractional ideal in $K$ generated by $1/v$. Assume $C^\infty_0(M,\mathcal{C}_{2}, \PI^1)=\frac{u}{\mathfrak{v}}, u \in C^\infty(M ,\C)$. Then the latter inherits a ring stucture by setting for any two elements $\frac{u_1}{v}, \frac{u_2}{v}, u_1, u_2 \in C^\infty(M ,\C)$ so that $u_1, u_2$ are not smoothly divisible by $v$ (that is there are no smooth functions $w_1, w_2\in C^\infty(M ,\C)$ so that $w_i=u_i/v, \ i=1,2$:
\begin{equation}\label{modmult}
\frac{u_1}{v}\cdot \frac{u_2}{v}= \frac{u_1\cdot u_2}{v} \in  C^\infty_0(M,\mathcal{C}_{2}, \PI^1).
\end{equation}
That is, the multiplicative structure on $C^\infty_0(M,\mathcal{C}_{2}, \PI^1)$ is inherited from $R$ by 'keeping the denominator stable'. Note that while we assume that at any $x_i \in \mathcal{C}_2$ the pole of any element of $\mathcal{M}^0_\infty(T^*M, \mathcal{C}_{2})$ is annihilated in the sense of (\ref{maximalideal}) by an element of a fixed finite power of $\mathfrak{m}_{x_i}$ at $x_i$, we do not denote this number explicitly. In addition, we will assume that $\mathcal{M}^0_\infty(T^*M, \mathcal{C}_{2})$ has actually also the structure of a 'principal fractional ideal', to be more precise, we denote for a multiindex $r=(r_1, \dots, r_{\tilde k}), \tilde k=|\mathcal{C}_{2}|, r_i \in \N^+$ by $\mathcal{M}^r_\infty(T^*M, \mathcal{C}_{2})$ the $\mathfrak{p}_2$-submodule of $\mathcal{M}^0_\infty(T^*M, \mathcal{C}_{2})$ so that there exists an element $1/\mathfrak{r}$ in the quotient field $K$ of $R=C^\infty(M, \C)$ so that $\mathcal{M}^r_\infty(T^*M, \mathcal{C}_{2})=\mathcal{M}^r(T^*M, \mathcal{C}_{2})/\mathfrak{r}$ and for any $x_i\in \mathcal{C}_{2}$ there exists an element $g_i \in \mathfrak{m}_{x_i}$ so that $g_i=0$ in $\mathfrak{m}_{x_i}/\mathfrak{m}_{x_i}^2$ and an element $\tilde r=(\tilde r_1, \dots ,\tilde r_{\tilde k}) \in \N^{\tilde k}$ with $\tilde r_i \leq r_i$ so that 
\begin{equation}\label{maximalideal2}
\Pi_{i=1}^k g_i^{\tilde r_i}\mathcal{M}^r_\infty(T^*M, \mathcal{C}_{2})\subset\mathcal{M}^r(T^*M, \mathcal{C}_{2}),
\end{equation}
that is the denominators in $\mathcal{M}^r_\infty(T^*M, \mathcal{C}_{2})$ have lower or equal maximal order $\tilde r_i$ as the minimal order of the vanishing ideals generating $\mathcal{M}^r(T^*M, \mathcal{C}_{2})$. If $\mathfrak{r}$ is such that each $g^{\tilde r_i}$ is neccessarily an element of $\mathcal{M}^r(T^*M, \mathcal{C}_{2})$ while (\ref{maximalideal2}) holds, we will call $r$ the order of the fractional ideal $\mathcal{M}^r_\infty(T^*M, \mathcal{C}_{2})$ and $\mathcal{M}^r(T^*M, \mathcal{C}_{2})$ the associated ring.
\\
We can consider $\mathcal{M}^0_\infty(T^*M, \mathcal{C}_{2})$ in the above sense as a $C^\infty(M, \C)$-submodule of $\mathcal{E}(T^*M, M\setminus \mathcal{C}_{2})$ by restriction and we will explain below how to define a $C^\infty_0(M ,\mathcal{C}_{2}, \PI^1)$-module structure on $\mathcal{M}^0_\infty(T^*M, \mathcal{C}_{2})$ with implied ring structure on $C^\infty_0(M ,\mathcal{C}_{2}, \PI^1)$ as indicated above. We can consider its 'complexification' $\mathcal{M}^0_\infty(T_{\C}^*M, \mathcal{C}_{2})$. We can define $\mathcal{M}^0_\infty(T^*_{\C}M, \mathcal{C}_{l})$ resp. $\mathcal{M}^r(T^*M, \mathcal{C}_{l})$ and $\mathcal{M}^r_\infty(T^*_{\C}M, \mathcal{C}_{l})$ analogously as $C^\infty_0(M ,\mathcal{C}_{l}, \PI^1)$ resp. $\mathfrak{p}^r_l$-submodules of $\mathcal{E}(T^*M, M\setminus \mathcal{C}_{l})$ on $M$. We then claim and prove below in Lemma/Assumption \ref{ass3} that $T_{l,2}$, after eventually modifying $T_{l,2}|{\rm span}(s_l)^\perp$ appropriately, induces 'by projecting out the singularity' a 'natural' map $\tilde T_{l,2}:\mathcal{M}^{r,T_{l,2}}_\infty(T^*_{\C}M, \mathcal{C}_{l}) \subset \mathcal{M}^r_\infty(T^*_{\C}M, \mathcal{C}_{l})\rightarrow \mathcal{M}^{\tilde r}_\infty(T_{\C}^*M, \mathcal{C}_{2})$ for an appropriate $\tilde r \in (\N^+)^{\tilde k}$ where $\mathcal{M}^{r,T_{l,2}}_\infty(T^*_{\C}M, \mathcal{C}_{l})$ is defined by replacing $\mathcal{M}^r(T^*_{\C}M, \mathcal{C}_{l})$ in (\ref{maximalideal2}) by $\mathcal{M}^{r,T_{l,2}}(T^*_{\C}M, \mathcal{C}_{l})$. On the other hand, since by assumption $T_{l,2}, \overline T^*$ (in (\ref{matrix1})) satisfy the first line of (\ref{global}), we have $\overline T^* \circ T_{l,2}(s_l) \subset \mathcal{M}^r_\infty(T^*_{\C}M, \mathcal{C}_{l})\subset \mathcal{M}_\infty^0(T^*_{\C}M, \mathcal{C}_{l})$ since $dS$ and $s_l$ both vanish exactly at $\mathcal{C}_l$ and the order of vanishing on $\mathcal{C}_l$ is determined by by $\hat s_l: M\rightarrow \hat P_{G/\tilde G}$. Finally note that we can restrict $\overline T^*$ in (\ref{matrix1}) to $\mathcal{M}^{\tilde r}_\infty(T_{\C}^*M, \mathcal{C}_{2})$.\\
We denote from now on by $W\cdot I^{\rm Gr}_n$ for some function $W \in C^\infty(M ,\mathcal{C}_{l,2}, \PI^1)$ the map $W\cdot I_{\mathcal{J}} \in {\rm End}(\mathcal{J})$ acting on $\mathcal{E}({\rm Gr}(T^\mathcal{J}_{\C}M^\pm))$ by pointwise identification ${\rm Gr}(T^\mathcal{J}_{\C}M^\pm)_x \simeq \mathfrak{h},\ x \in M$ where $W$ acts on an element of $T \in \mathfrak{h}$ (understood as a complex symmetric $n \times n$-matrix) by scalar multiplication and interpreted as a fibrewise symplectic mapping $W \cdot I^{\rm Gr}_n: \mathcal{E}(T^*M, M\setminus \mathcal{C}_{l,2}) \rightarrow \mathcal{E}(T^*M, M\setminus \mathcal{C}_{l,2})$ covering he identity on $M$ using the correspondence discussed above (\ref{matrix1}), we conclude that wrt some chosen local $O(n)$-frame adopted to the given Lagrangian distribution $\Lambda\subset T^*M$ and for a local section $\hat s: U\subset M\rightarrow P_{\tilde G}$ and $W \in C^\infty(M ,\mathcal{C}_{l,2}, \PI^1)$ we have the following multiplicative map
\begin{equation}\label{ast}
W\ast := W \cdot I^{\rm Gr}_n: \mathcal{E}(T^*M, M\setminus \mathcal{C}_{l,2}) \rightarrow \mathcal{E}(T^*M, M\setminus \mathcal{C}_{l,2}),\quad s \mapsto \left(\begin{smallmatrix} \sqrt{W}I_{n}& 0\\0&\sqrt{W}^{-1}I_{n}\end{smallmatrix}\right )s,
\end{equation}
where $I_{n}=Id_{\R^n}$, note that $\ast$ thus depends on the choice of $P_{\tilde G}\subset P$ and in especially on $J$, we will notationally suppress this dependency in the following when the context allows it. Thus, if we compose $\kappa_J:TM\simeq T^\mathcal{J}_{\C}M^+$ with $\omega$-duality the pointwise multiplication of $T \in \mathfrak{h}$ with $W \in C^\infty(M ,\mathcal{C}_{l,2}, \PI^1)$ corresponds as a (fibrewise, real) symplectic mapping to the multiplication $W\ast$ and preserves evidently $\mathcal{E}^0(T^*M, M\setminus \mathcal{C}_{l,2})$. This of course does not give a module structure on $\mathcal{E}(T^*M, M\setminus \mathcal{C}_{l,2})$. However, consider $C^\infty(M ,\mathcal{C}_{l,2}, \PI^1)$ as as subset of the quotient field $K$ of $R=C^\infty(M)$. Then any element $W\in C^\infty(M ,\mathcal{C}_{l,2}, \PI^1)$ generates a principal thus invertible fractional ideal $\mathfrak{w}=WR$ in $K$. Let $\mathfrak{w}^{-1}$ be its inverse, it is generated by $W^{-1}$ in $K$ (thus $\mathfrak{w}^{-1}=W^{-1}R$). We then define a $C^\infty(M)$-module structure on $\mathcal{E}(T^*M, M\setminus \mathcal{C}_{l,2})$ in the following sense:
\begin{equation}\label{ast2}
\ast_W:  C^\infty(M)\times \mathcal{E}(T^*M, M\setminus \mathcal{C}_{l,2}) \rightarrow \mathcal{E}(T^*M, M\setminus \mathcal{C}_{l,2}),\quad (g,s) \mapsto g.(W\ast s),
\end{equation}
where here, $g.s$ means simply the usual $C^\infty(M)$-operation of on $\mathcal{E}(T^*M)$. Assume now we have two sub-modules $\mathcal{E}^W_{1,2}(\mathcal{C}_{l,2})\subset \mathcal{E}(T^*M, M\setminus \mathcal{C}_{l,2})$ so that
\[
\ast_{w}(u,\mathcal{E}^W_{1}(\mathcal{C}_{l,2})) = \mathcal{E}^W_{2}(\mathcal{C}_{l,2}),\ w \in \mathfrak{w}^{-1}, u \in C^\infty(M), \quad \ast_{w}(u,\mathcal{E}^W_{2}(\mathcal{C}_{l,2})) = \mathcal{E}^W_{1}(\mathcal{C}_{l,2}),\ w \in \mathfrak{w}, u \in C^\infty(M).
\]
Let now $\mathcal{E}^W_{1\oplus 2}(\mathcal{C}_{l,2})= \mathcal{E}^W_{1}(\mathcal{C}_{l,2}) \oplus \mathcal{E}^W_{2}(\mathcal{C}_{l,2})$, this is a $C^\infty(M)$-module by the usual component-wise multiplication (note that we do not assume that $ \mathcal{E}^W_{1,2}(\mathcal{C}_{l,2})$ are disjoint as subspaces of $\mathcal{E}(T^*M, M\setminus \mathcal{C}_{l,2})$). Then we define for $W \in \mathfrak{w}^{-1}$ a $C^\infty(M)$-module structure on $\mathcal{E}^W_{1\oplus 2}(\mathcal{C}_{l,2}))$ by
\begin{equation}\label{ast3}
\tilde \ast_W: C^\infty(M)\times \mathcal{E}^W_{1\oplus 2}(\mathcal{C}_{l,2})\rightarrow \mathcal{E}^W_{1\oplus 2}(\mathcal{C}_{l,2}), \quad \  \tilde \ast_W (u, \begin{pmatrix}s_1\\ s_2\end{pmatrix})= \begin{pmatrix}u.{\bf 1}& 0\\ \ast_{W}(u, \cdot)& u.{\bf 1}\\ \end{pmatrix}  \begin{pmatrix}s_1\\ s_2\end{pmatrix}.
\end{equation}
In the following we will shortly write $(\frac{u}{w})\ast(\cdot):=\tilde \ast_w(u, \cdot)$ for the implied $\mathfrak{w}^{-1}$-module structure on $\mathcal{E}^W_{1\oplus 2}(\mathcal{C}_{l,2})$ if $w \in \mathfrak{w}^{-1}$ as above and $u \in C^\infty(M, \C)$ where as remarked $\mathfrak{w}^{-1}$ carries the ring structure induced from $C^\infty(M, \C)$ as constructed in (\ref{modmult}).\\
In the following, let for $W \in \mathfrak{w}$ be $\mathcal{E}^W_{2}(\mathcal{C}_{l,2})=\mathcal{M}^r_\infty(T^*M, \mathcal{C}_{2})=\mathcal{M}^r(T^*M, \mathcal{C}_{2})/\mathfrak{w}$ and $\mathcal{E}^W_{1}(\mathcal{C}_{l,2})=\mathcal{M}^r(T^*M, \mathcal{C}_{2})$, then obviously $\ast_{W^{-1}}(\cdot, \mathcal{E}^W_{1}(\mathcal{C}_{l,2})) = \mathcal{E}^W_{2}(\mathcal{C}_{l,2})$ and $\ast_{W}(\cdot, \mathcal{E}^W_{2}(\mathcal{C}_{l,2})) = \mathcal{E}^W_{1}(\mathcal{C}_{l,2})$ ($\mathcal{C}_{l,2}$ stands here for any of the subsets $\mathcal{C}_{l}, \mathcal{C}_{2}$ of $M$). We then define 
\begin{equation}\label{maximalnew}
\tilde {\mathcal{M}}^{r}_\infty(T^*_{\C}M, \mathcal{C}_{l,2}):=\mathcal{E}^W_{1\oplus 2}(\mathcal{C}_{l,2})
\end{equation}
according to the conventions above, analogously $\tilde {\mathcal{M}}^{r, T}_\infty(T^*_{\C}M, \mathcal{C}_{l,2})$ and we will thus implicitly always use for any given $W \in C^\infty_0(M ,\mathcal{C}_{l,2}, \PI^1)$ the multiplications $\tilde \ast_W$ (resp. $\tilde \ast_{W,J}$ to emphasize its dependency on $J$) in the diverse module structures $\tilde {\mathcal{M}}^{r, T}_\infty(T^*_{\C}M, \mathcal{C}_{l}), \tilde {\mathcal{M}}^{r}_\infty(T^*_{\C}M, \mathcal{C}_{l}), \tilde {\mathcal{M}}^{r, T}_\infty(T^*_{\C}M, \mathcal{C}_{2}), \tilde {\mathcal{M}}^{r}_\infty(T^*_{\C}M, \mathcal{C}_{2})$ as subsets of $\mathcal{E}^2( \mathcal{C}_{l,2}):=\mathcal{E}^0(T^*M, M\setminus \mathcal{C}_{l,2})\oplus \mathcal{E}^0(T^*M, M\setminus \mathcal{C}_{l,2})$ over $\tilde R:= C^\infty_0(M ,\mathcal{C}_{l,2}, \PI^1)$, where we endow the latter with the ring structure it inherits from its structure as a principal fractional ideal in $K$ as in (\ref{modmult}). To be more precise we define
\begin{Def}\label{blaJext}
$\tilde {\mathcal{M}}_\infty^{r,T}(T^*_{\C}M, \mathcal{C}_{l,2})$, $r\in (\N^+)^k$ is the $\tilde R$-submodule of $\mathcal{E}^2( \mathcal{C}_{l,2})$ defined by (\ref{maximalideal}), (\ref{maximalideal2}) resp. (\ref{laurent}), $1/\mathfrak{w} \subset K$ and $T \in {\rm End}^r_\omega(T^*M\setminus \pi^{-1}(\mathcal{C}_{l}))$ as above, while replacing the $R$-multiplication in (\ref{maximalideal2}), (\ref{maximalideal2}) resp. (\ref{laurent}) by the $\tilde R$-multiplication $\tilde \ast_{W,J}$ as defined in (\ref{ast3}). {\it In addition}, since $\tilde {\mathcal{M}}^{r,T}_\infty(T^*_{\C}M, \mathcal{C}_{l})$ defined in this way is in general not $J$-invariant for a given $\nabla$-parallel almost complex structure associated to the $\tilde G$-reduction $P_{\tilde G}\subset P$, we require $\tilde {\mathcal{M}}^{r,T, J}_\infty(T^*_{\C}M, \mathcal{C}_{l,2})\subset \mathcal{E}^2(\mathcal{C}_{l,2})$, $r\in (\N^+)^k$ to be the smallest submodule of $\mathcal{E}^2(\mathcal{C}_{l,2})$ that contains $\mathcal{M}^{r,T}_\infty(T^*_{\C}M, \mathcal{C}_{l,2})$, $r\in (\N^+)^k$ and is left invariant by $J$. 
\end{Def}
Note that since $J\circ W\ast_J=W^{-1}\ast_J$, we have necessarily that $\mathcal{M}^{r,T, J}_\infty(T^*_{\C}M, \mathcal{C}_{l})$ is of the form $\mathcal{M}^{r,T, J}_\infty(T^*_{\C}M, \mathcal{C}_{l})=\mathcal{M}^{r,T, J}(T^*_{\C}M, \mathcal{C}_{l})/\mathfrak{r}$ for $\mathfrak{r} \in \mathcal{M}^{r,T, J}_\infty(T^*_{\C}M, \mathcal{C}_{l})\subset \mathcal{E}^0(T^*M, M\setminus \mathcal{C}_{l})\cap \mathcal{E}(T^*M, M)$ an appropriate ideal reflecting the set of 'minimal vanishing orders'. We will suppress notationally in the following in general the dependence of $\mathcal{M}^{r,T, J}_\infty(T^*_{\C}M, \mathcal{C}_{l})$ resp. $\mathcal{M}^{r,T, J}(T^*_{\C}M, \mathcal{C}_{l})$ on $J$ as long as it is clear which almost complex structure i.e. $\tilde O(n)$-reduction $P_{\tilde G}$ of $P$ is involved in its definition (note that $\ast$ in fact depends not only on $J$, but on $\Lambda_{\tilde G}$ and the chosen $P_{\tilde G}$). To proceed we need the following observation/assumption, it reflects the fact that associated to a standard, irreducible, in general singular Frobenius structure there are associated 'natural denominators' that can be 'switched on and off' in a sense ('projecting/not projecting out the singularity' of $T_{l,2}$ and $\overline T^*$).
\begin{asslemma} \label{ass3}
Consider the classification of irreducible standard (singular) Frobenius structures on Proposition \ref{classi} resp. Theorem \ref{genclass} and assume that a given pair of (nonsingular) maps $\hat s^0_l, \hat s^0_2: M\rightarrow \hat P_{G/\tilde G}$, not necessarily transversal to the zero section of $T^*M$, defines a pair of {\it dual} (generalized standard irreducible, in general singular) Frobenius structures with underlying almost complex structures $J_0, J_{T,0}$ respectively in the sense of Theorem \ref{theorem2}, i.e. is homotoped in the sense of Proposition \ref{classi} to define a pair $\hat s_l, \hat s_2: M\rightarrow \hat P_{G/\tilde G}$ of singular irreducible standard Frobenius structures. Let now be $\tau \in \{0,1\}$. Then the above discussion gives well-defined maps 
\begin{equation}
\tilde T_{l,2}:\tilde{\mathcal{M}}^{r,T_{l,2}}_\infty(T^*_{\C}M, \mathcal{C}_{l})\rightarrow \left\{\begin{matrix}\mathcal{M}^{\tilde r}(T_{\C}^*M, \mathcal{C}_{2})\oplus \{0\},\  \tau = 0\\
\{0\} \oplus \mathcal{M}^{\tilde r}_\infty(T_{\C}^*M, \mathcal{C}_{2}),\ \tau=1\end{matrix}\right.  \quad \subset \tilde {\mathcal{M}}^{\tilde r}_\infty(T_{\C}^*M, \mathcal{C}_{2}),
\end{equation}
for appropriate $r \in (\N^+)^k, \tilde r \in (\N^+)^{\tilde k}$, where the denominators determining $\mathcal{M}^{r,T_{l,2}}_\infty(T^*_{\C}M, \mathcal{C}_{l})$ and $\mathcal{M}^{\tilde r}_{(\infty)}(T_{\C}^*M, \mathcal{C}_{2})$ are given by $|s_l|^2$ and $|s_2|^{2\tau}$ respectively, where the norms are induced by $(\omega,J_0)$ and $(\omega, J_{T,0})$, respectively and analogously
\begin{equation}
\overline T^*:\tilde{\mathcal{M}}^{\tilde r, T^*}_\infty(T_{\C}^*M, \mathcal{C}_{2}) \rightarrow \left\{\begin{matrix} \mathcal{M}^r(T^*_{\C}M, \mathcal{C}_{l}) \oplus \{0\}, \ \tau=0\\ \{0\} \oplus \mathcal{M}^r_\infty(T^*_{\C}M, \mathcal{C}_{l}), \ \tau=1 \end{matrix}\right. \quad \subset \tilde {\mathcal{M}}^r_\infty(T^*_{\C}M, \mathcal{C}_{l}),
\end{equation}
where the denominators determining $\mathcal{M}^{r,T_{l,2}}_{(\infty)}(T^*_{\C}M, \mathcal{C}_{l})$ and $\mathcal{M}^{\tilde r}_\infty(T_{\C}^*M, \mathcal{C}_{2})$ are given by $|s_l|^{2\tau}$ and $|s_2|^2$ respectively i.e. $\tilde T_{l,2} \in {\rm End}^r_\omega(T^*M\setminus \pi^{-1}(\mathcal{C}_{l}))$ and $\overline T^*\in {\rm End}^{\tilde r}_\omega(T^*M\setminus \pi^{-1}(\mathcal{C}_{2}))$, in fact in the transversal case we infer $r= {\bf 1}, \tilde r= {\bf 1}$. In the following, we will consider only (pairs of) 'singular' irreducible (generalized) standard Frobenius structures $\hat s_l, \hat s_2: M\rightarrow \hat P_{G/\tilde G}$ that are  endpoints of an (on $M\setminus \mathcal{C}_{2}$ resp. $M\setminus \mathcal{C}_{l}$) smooth homotopy of smooth maps in the sense of (the proofs of) Proposition \ref{classi} and Theorem \ref{genclass} intersecting the zero section of $T^*M$ transversally. We finally note that the set $C^\infty_0(M ,\mathcal{C}_{2}, \PI^1)$ (analogously $C^\infty_0(M ,\mathcal{C}_{l}, \PI^1)$) can be canonically identified with the fractional ideal $\mathfrak{p}(\mathcal{C}_{2})/\mathfrak{r}$ in the quotient field $K$ of $R=C^\infty(M, \C)$ where the ideal $\mathfrak{p}(\mathcal{C}_{2})$ of $C^\infty(M ,\C)$ is generated by $\prod_{x \in \mathcal{C}_2}{\mathfrak{m}}_x$, ${\mathfrak{m}}_x$ maximal ideal at $x\in M$, and $1/\mathfrak{r} \in K$ resp. $\mathfrak{p}(\mathcal{C}_{2})/\mathfrak{r}$ is an element of the (generalized) ideal quotient in $K$ of the ideals ${\bf 1}$ resp. $\mathfrak{p}(\mathcal{C}_{2})$ in the ring $R$.
\end{asslemma}
\begin{proof}
We have to show that we can choose $\tilde T_{l,2}, \overline T^*$ as described above so that their images lie in the $C_0^\infty(M,\C)$- submodules of $\mathcal{E}^0(T^*M, M\setminus \mathcal{C}_{l})$ resp. $\mathcal{E}^0(T^*M, M\setminus \mathcal{C}_{2})$ generated by the sum of the (fractional) ideals $\mathfrak{p}(\mathcal{C}_{l})$, $\mathfrak{p}(\mathcal{C}_{l})^{-1}$ and $\mathfrak{p}(\mathcal{C}_{2})$, $\mathfrak{p}(\mathcal{C}_{2})^{-1}$, respectively. But that will follow in the case of $\overline T^*$ from the definition of the homotoped almost complex structures $J$ defining a singular Frobenius structure that was defined in Proposition \ref{classi}, where the homotopy blows up at $\mathcal{C}_l$ resp. $\mathcal{C}_2$ in the case of the given maps $\hat s_l, \hat s_2: M\rightarrow \hat P_{G/\tilde G}$. Here if $\xi_1={\rm pr}_1\circ \hat s_l$, we defined the homotopy as $t \mapsto J_t, t \in [0,1]$ where $J_t=\frac{1}{(1-t)+t|\xi|^2}J$ on ${\rm ker}(\xi_1)^\perp$ and $J_t=((1-t)+t|\xi|^2)J$ on $J\circ{\rm ker}(\xi_1)^\perp$. Now since $|\xi|^2:M \rightarrow \R$ is differentiable on $M$, it is thus an element of $C_0^\infty(M ,\mathcal{C}_{l}, \PI^1)$, analogously for $\frac{1}{|\xi|^2}:M \rightarrow \PI$, analogously of course we can argue for $\hat s_2$. Then the pole order of $T_{l,2}|{\rm span}_{\C}(s_l)$ constructed in the sense of (\ref{matrix1}) at $x_i$ is determined by the pole order of $\frac{1}{|\xi|^{\tau}}:M \rightarrow \PI$ in the sense of the formal Laurent expansion associated to (\ref{laurent}) which in turn coincides with the order of vanishing of $|s_l|: M\rightarrow \R$ and thus with the order of vanishing of the elements of $s_l$ in the sense of (\ref{laurent}), by the Morse Lemma we can thus chose $r={\bf 1}$. On the other hand, $T_{l,2}|(s_l)^\perp$ (where $(\cdot)^\perp$ here refers to the pointwise $\omega(\cdot,J \cdot)$-orthogonal complement) can be always chosen to be the identity outside a neighbourhood $U_l$ of $\mathcal{C}_l$ and $U_2$ of $\mathcal{C}_2$ and the identity, composed with its restriction to any given one-dimensional subspace of $(s_l)^\perp$ (see the proof of Theorem \ref{theorem3} below) multiplied with appropriate elements of a sufficiently great (positive) power of $\mathfrak{p}(\mathcal{C}_{l})$ on $U_l$ and elements of appropriate positive (again, sufficiently great) of $\mathfrak{p}(\mathcal{C}_{2})$ on $U_2$ from which the claim for $T_{l,2}$ and $\tau=1$, namely $\tilde T_{l,2}(\mathcal{M}^{r,T_{l,2}}_\infty(T^*_{\C}M, \mathcal{C}_{l}))\subset \mathcal{M}^{\tilde r}_\infty(T_{\C}^*M, \mathcal{C}_{2})$ follows already, if we take into account that ${\rm pr}_1\circ \hat s_l$ and ${\rm pr}_1\circ \hat s_2$ in itself are smooth on $M$. A similar argument holds for $\overline T^*|\mathcal{M}^{\tilde r, T^*}_\infty(T_{\C}^*M, \mathcal{C}_{2})$, where $T^*$ acts on $T_{\C}^*M$ mapping $T_{\C}^*M^{\pm,J}\simeq T^*M$ bijectively onto $T_{\C}^*M^{\pm,T^*.J}\simeq T^*M$ ($J$ acting on $T^*M$ via $\omega$-duality) and $J$ is the almost complex structure corresponding to the $\hat U(n)$-reduction of $P$ induced by $\pi_{Mp}(\hat s_l)$. Note finally that in the case of a Kaehler manifold, we can choose an $\omega$-compatible almost complex structure $J$ which is parallel wrt the given symplectic connection $\nabla$ and an associated Frobenius structure $\hat s_l: M\rightarrow \hat P_{G/\tilde G}$ which is non-singular by Proposition \ref{classi}, thus the claimed trait of $\overline T^*$ follows trivially by smoothness.
\end{proof}
For the following recall the definition of the Koszul-bracket $[\cdot, \cdot]_{J, \Lambda_{\tilde G}}: \Gamma(T^*M)^2 \rightarrow \Gamma(T^*M)$ for a given $\omega$ compatible almost complex structure $J$ in Definition \ref{koszul} and a chosen $\hat O(n)$-reduction $P_{\tilde G}$ of $P^J$. we will below assume that we have chosen a symplectic connection $\nabla$ reducing to $P_{\tilde G}$, i.e. $\nabla J=0$ and $\nabla$ preserves the Lagrangian distribution $\Lambda_{\tilde G}: M\rightarrow {\rm Lag}(T^*M, \omega)$ defining $P_{\tilde G}$. Also, the almost complex structures $(J, J_T)= (J_1, J_{T,1})$ implicit in the objects ($R$-modules in th sense of Definition \ref{blaJext}) defined below will always be those associated to the primordial sections $\hat s_l, \hat s_2: M\rightarrow \hat P_{G/\tilde G}$ to which these objects are associated in the sense described above, if not remarked otherwise. Also, we will in the following denote by $J_0$ resp. $J_{T,0}$ the {\it non-singular} almost complex structures associated to $\hat s^0_l, \hat s^0_2: M\rightarrow \hat P_{G/\tilde G}$ in the sense of Assumption \ref{ass3}, that is before the endpoint-singular homotopy described in Proposition \ref{classi} resp. Theorem \ref{genclass} is applied, the corresponding map in the second member of (\ref{matrix1}) is denoted by $\overline T^*_0$. Recall also the notion of symplectic reduction of a symplectic vector space $(V, \omega)$ art a coisotropic subspace $W\subset Ann(W)$, where $Ann(W)=\{u\in V, \omega(u,v)=0 {\rm \ for\ all}\ v \in W\}$. Then $W/Ann(W)$ is a symplectic vector space wrt the projected symplectic form $\omega$. Further, for a Lagrangian subspace $L\subset V$ we have that $R_W(L)=(L\cap W)/Ann(W)$ is Lagrangian in $W/Ann(W)$.\\
Let now $\pi_{\mathcal{J}}: \mathcal{J}\rightarrow M$ denote the twistor bundle bundle of $\omega$-compatible almost complex structures on $M$. Relative to a fixed $\omega$-compatible almost complex structure $J$, inducing a $\hat U(n)$-reduction of a chosen metaplectic $P$ we can identify $\mathcal{J}\simeq P_{\hat U(n)} \times_{\rho,Ad}Sp(2n)/U(n)$. We note (compare the discussion below (\ref{dia1})) that for any fixed smooth section $J: M\rightarrow \mathcal{J}$, any auxiliary section $J_T: M\rightarrow \mathcal{J}$ in this sense gives rise to a map $T\in  {\rm End}^0_\omega(T^*M)$ or more generally, in the above discussed sense, our sections $\hat s_l, \hat s_2: M\rightarrow \hat P_{G/\tilde G}$ give rise to elements $T_{l,2} \in {\rm End}^r_\omega(T^*M\setminus \pi^{-1}(\mathcal{C}_{l}))$ resp. $\overline T^* \in {\rm End}^r_\omega(T^*M\setminus \pi^{-1}(\mathcal{C}_{2}))$. By complex linear continuation to $T^*_{\C}M$ and fibrewise projectivization, that is denoting the fibrewise projectivization of $T^*_{\C}M$ by $\PI^*M,$ the latter two can be interpreted as fibrewise symplectic endomorphisms in the bundle $\PI^n$, to be denoted by ${\rm End}^r_\omega(\PI^*M,\mathcal{C}_{l})$ resp. ${\rm End}^r_\omega(\PI^*M,\mathcal{C}_{2})$. Let now $c_1(L)\in H^2(\PI^n, \C)$ (cf. \cite{gilkey}) be the first Chern form induced by the tautologial line bundle $L$ over $n$-dimensional complex projective space $\PI^n$, i.e if $s=(z_0, \dots, z_n)$ locally over $U\subset\PI^n$, i.e.
\[
c_1(L)=\frac{i}{2\pi}\partial \overline \partial (1+|z|^2),
\]
not that $c_1(L)$ is invariant under the action of $U(n)$ on $\PI^n$. Globalizing this construction to $\PI^*M$, assuming that $T^*M$ and thus $\PI^*M$ is equipped with the canonical symplectic structure $\tilde \omega$ and a nearly complex structure $\tilde J$ that is compatible with the one on $M$ in the sense that $s_0^*\tilde \omega(\cdot, \tilde J\cdot)=\omega(\cdot J, \cdot)$, symplectic connection $\tilde \nabla$ satisfying $\tilde \nabla \tilde J=0$ and the existence of a $\tilde \nabla$-invariant horizontal Lagrangian polarization $\mathcal{H}\subset T(T^*M)$ wrt $\tilde \omega$ we arrive at a complex line bundle $\mathbb{L}\rightarrow \PI^*M$ and corresponding Chern class $c_1(\mathbb{L},\PI^*M) \in H^2(\PI^*M, \C)$ by defining over any open subset $U\subset M$ and a given local section $s_U:U\rightarrow P_{\hat U(n)}$  trivializing $\PI^*M|U$ by duality $\mathcal{L}|U=L\times (\PI^n\times U)$ and the complex two forms $c_1(\mathbb{L},\PI^*M)|(V\PI^*M|U)=c_1(\mathcal{L}|U)$ considering the observed $U(n)$-invariance of $c_1(L)$, here $V\PI^*M\subset T\PI^*M$ being the vertical tangent bundle of $\PI^*M$ and setting $c_1(\mathbb{L},\PI^*M)=0$ on $\mathcal{H}$. Since $\mathcal{H}$ is $\tilde \nabla$-parallel, this prodecure gives a well-defined element of $H^2(\PI^*M, \C)$, denoted by $c_1(\mathbb{L},\PI^*M)$. For two given smooth (possibly singular at $\mathcal{C}_2$) sections $J, J_T: M\rightarrow \mathcal{J}$ satisfying $\nabla J=\nabla J_T=0$ and the associated $\overline T^* \in {\rm End}^r_\omega(T^*M\setminus \pi^{-1}(\mathcal{C}_{2}))$ we note that the latter acts on $\mathbb{L}\subset T^*(\PI^*M)$ by pullback (for the latter inclusion, we refer to \cite{gilkey}, Lemma 2.3.2), that is maps $\mathbb{L}$ to $\overline T^*\mathbb{L}\subset T^*(\PI^*M)$, we associate an element $c_1(\PI^*M, J, J_T, \overline T^*)=c_1(\overline T^*\mathbb{L}, \PI^*M) \in H^2(\PI^*M, \C)$. Analogously for two closed sections $s_l, s_2: M\rightarrow T^*M$ we associate with $T_{l,2} \in {\rm End}^r_\omega(T^*M\setminus \pi^{-1}(\mathcal{C}_{l}))$ the element $c_1(\PI^*M, s_1, s_2, \overline T_{l,2})= c_1(\overline T_{l,2}^*\mathbb{L}, \PI^*M) \in H^2(\PI^*M, \C)$. Note that since $\overline T_{l,2}$ is a priorily not uniquely determined by $s_1, s_2$, the class $c_1(\PI^*M, s_1, s_2, \overline T_{l,2})$ still waits for a correct definition, the ambiguity will be fixed in the proof of Theorem \ref{theorem3} below. Finally, for a closed section $s:M\rightarrow T^*M$, we can look at its projectivization $s_{\Pi}:M\rightarrow \PI^*M$ (here $\PI^*M$ should be understood as the bundle that arises from the fibrewise projective completion $\PI(T_x^*M \oplus \C)$ of $T^*_{\C}M$) and define two elements of $\Omega^2(M, \C)$ by
\[
\mathfrak{c}_1(M, s,\overline T^*)= s^*_{\Pi}c_1(\PI^*M, J, J_T, \overline T^*), \quad \mathfrak{c}_1(M, s, \overline T_{l,2})=s^*_{\Pi}c_1(\PI^*M, s_1, s_2, \overline T_{l,2}).
\]
Note that both $\overline T_{l,2}$ and $\overline T^*$ still wait for a precise (unambiguous) definition, which will be given in the proof of Theorem \ref{theorem3}. We can state nonetheless
\begin{prop}\label{blachern}
Let $s:M\rightarrow T^*M$ be a closed section and assume the existence of a $\tilde \nabla$-invariant horizontal Lagrangian polarization $\mathcal{H}\subset T(T^*M)$ for a given $\hat U(n)$-reduction $P_{\hat U(n)}$ of $P$ and a given symplectic connection satisfying $\nabla J=0$. Assume we have a section $T \in \mathcal{E}(T^*M, M)$ such that $dT=0$ where $T$ is considered as an element of $\Omega^2(M, TM)$. Then we have that the form $\mathfrak{c}_1(M, s,T)=s_{\Pi}^*c_1(\overline T^*\mathbb{L}, \PI^*M) \in \Omega^2(M,\C)$ is closed and defines thus a class in $H^2(M, \C)$. In our concrete situation described above, assume that $s_l \in \mathcal{E}^0(T^*M, M\setminus \mathcal{C}_{l})$, $s_2 \in \mathcal{E}^0(T^*M, M\setminus \mathcal{C}_{2})$ are both closed, with $s_l$ $C^1$-small and $s_l, s_2$ intersecting the zero section of $T^*M$ transversally. Then $\mathfrak{c}_1(M, s_l,\overline T^*)$ and $\mathfrak{c}_1(M, s_2, \overline T_{l,2})$ define integral cohomology classes, that is elements of $H^2(M, \Z)$.
\end{prop}
\begin{proof}
Consider an open cover $\mathcal{U}$ of $M\setminus \mathcal{C}_{l,2}$ so that for any $U\in \mathcal{U}$ we have that $P_{\hat U(n)}|U$ is trivial, that is $P_{\hat U(n)}|U\simeq U\times \hat U(n)$. If $\pi:\PI^*M\rightarrow M$ is the canonical projection, consider $\overline T^*\mathbb{L}|\pi^{-1}(U)$. Consider over $\pi^{-1}(U)$ the trivializing open sets for $\mathbb{L}$, of the form $\tilde U_j=U \times (\PI^n\cap V_j), j\in \{1,\dots, n\}$, where $V_j\subset \PI^n$ is the open set over which there exists the local coordinate system $(z^j_0,\dots, z_n^j)$ obtained by $z^j_k=z_k/z_j$ and deleting $1=z_j/z_j$. Over any of the 
$\tilde{U}_j$, we can define the $1$-form
\[
\alpha_{U_j}=\sum_{i=1}^n \frac{z_j d\overline z_j}{1+|z|^2}.
\]
Note that since $\overline T^*$ (as well as $T_{l,2}$) are pointwise isomorphisms of the fibres of $\PI^*M$, and this isomorphism decends to the fibres of $L$ over $M\setminus \mathcal{C}_{l,2}$, we can consider $\tilde z_j= T^*(z_j)$, analogously for $T_{l,2}$, as a set of coordinate systems of $\pi^{-1}(U)$ with respective charts $(\overline T^*(\tilde U_j), \tilde z_j)$ resp. $(T_{l,2}(\tilde U_j), \tilde z_j)$, we will use the same notation for $\tilde z_j$ and $z_j$. In this sense, the forms $(\overline T^*)^*(\alpha_{\tilde U_j})$ assemble to a global one form $\tilde \alpha$ on $\mathcal{P}=\pi^*(P_{\hat U(n)})$, note that $\mathcal{P}$ is a $\hat U(n)$-bundle over $\PI^*(M\setminus \mathcal{C}_{l,2})$, which extends by continuity to $\PI^*M$. It is then easy to see that $\tilde \alpha$ descends to a $\C/\Z$-valued one-form $\alpha$ on $\PI^*M$ (where $\Z$ is here the subring of real integers in $\C$) and that $s^*\alpha$, where $s$ is as in the statement of the proposition, is a primitive for $\pi_{\C/\Z}\circ \mathfrak{c}_1(M, s,\overline T^*)$, where $\pi_{\C/\Z}: \C \rightarrow \C/\Z$ is the canonical projection. This already gives the assertion for the class $\mathfrak{c}_1(M, s,\overline T^*)$ defined by $\overline T^*$, referring to the general theory of differential characters, cf. \cite{cheegersimons}. Of course the above arguments also show that $\mathfrak{c}_1(M, s,\overline T^*)$ is closed. The case of $\mathfrak{c}_1(M, s, \overline T_{l,2})$ is analogous. 
\end{proof}
Note that a closed section $s \in \mathcal{E}(T^*M, M\setminus \mathcal{C}_l)$ (with zero locus $\mathcal{C}_l$) and an element $J \in \mathcal{J}$ define a real $2$-plane section $E_{s, j}\subset T^*(M\setminus \mathcal{C}_l)$ by setting $E_{s, J}=span_{\R}(s, J^*s)$ over $M\setminus \mathcal{C}_l$ and thus a trivialization of a complex $1$-dimensional subbundle $L_{s, J}\subset \PI^*(M\setminus \mathcal{C}_l)$. Over each $x \in M\setminus \mathcal{C}_l$, $L_{s, J}$ spans a $1$-cell in $\PI^n$ wrt to the cell decomposition of $\PI^n$ into $n+1$ different cells of real dimension $2k, 0\leq k\leq n$ (that is, $E_{s, J}$ fixes a reduction of $P$ consisting of exactly those unitary frames whose first two basis vector coincides with the pair $(s, J^*s)$ spanning $E_{s,J}$, $L_{s,J}$ is thus a trivial line bundle over $M\setminus \mathcal{C}_l$ and we have $\PI^*(M\setminus \mathcal{C}_l)\simeq L_{s, J} \oplus \PI_{n-1}^*(M\setminus \mathcal{C}_l)$, where $\PI_{n-1}^*(M\setminus \mathcal{C}_l)$ has fibre $\PI^{n-1}$ over $M\setminus \mathcal{C}_l$ (but is in general non-tivial). We can then define a projection $\pi_{S}:\PI^*(M\setminus \mathcal{C}_l)\rightarrow SL_{s, J} \oplus \PI_{n-1}^*(M\setminus \mathcal{C}_l)$, where $SL_{s, J}$ is the unit sphere $S^1$-bundle in $L_{s, J}$, by defining locally $\pi_{S}((x, (z, v))= (x, (\pi_{S\PI^1}(z), v)), x \in M, v \in \PI^{n-1}$, where $\pi_{S\PI^1}: \PI^{1}\rightarrow S^1$ is the projection onto the equator $S=S^1\subset \PI^{1}\simeq S^2$. Note that since $H^*(\PI^{n},\C)\simeq \C\oplus 0\oplus \C \dots \oplus \C$ vanishes in uneven dimensions, by the long exact cohomology sequence for the pair $(\PI^n, S)$ (considering $S\subset \PI^1\subset \PI^n$), we can infer an isomorphism $H^q(\PI^n, S, \C)\simeq H^{q-1}(S, \C), 1\leq q\leq 2n$. We denote the generator of $H^1(S^1, \C)$ by $\eta=d\theta$, where $\theta: S^1 \rightarrow \R/\Z$ is the angular coordinate on $S^1$, and define the pullback of the (vertical) one-form $\eta \oplus 0 \in \Omega^1(SL_{s, J} \oplus \PI_{n-1}^*(M\setminus \mathcal{C}_l), \C)$ by $\pi_S$ as $\tilde \eta_s=\pi_S^*(\eta_s) \in \Omega^1(\PI^*(M\setminus \mathcal{C}_l), \C)$ (note that again $\eta \in H^1(S^1, \C)$ here is understood to be globalized to a $1$-form $\eta_s$ on $SL_{s, J}$ by choosing a horizontal distribution $\mathcal{H}\subset TL_{s, J}$ which is parallel wrt a given symplectic connection $\nabla$ and setting $\eta_s(v)=0\ \forall v \in \Gamma(\mathcal{H})$). Then the fact that $d\tilde \eta_s=0$ and in fact $\tilde \eta_s\in H^1(\PI^*(M\setminus \mathcal{C}_l), \Z)$ follows quite analogously as in in the proof of Proposition \ref{blachern} which we will not formalize here consequently. By the above construction, we have moreover that 
\[
\hat \eta_s:= s^*\tilde \eta_s \in H^1(M\setminus \mathcal{C}_l, \Z)
\]
is well-defined, we will address in the proof of Theorem \ref{theorem3} below the question in which sense we have that $\hat \eta_s$ extends to an integral element of $H^1(M, \C)$.\\
Note that for any module $M$ over a ring $R$ and given a fractional ideal $\mathfrak{k}=\mathfrak{r}/r\subset K$ with $r \in R$ and an ideal $\mathfrak{r}\subset R$ and $K$ the quotient field of $R$, we can regard $\mathfrak{k}$ as a ring by inheriting the multiplication in $\mathfrak{r}\subset R$. Thus we can regard the localization $M_{(0)}$ of $M$ at $(0)$ as a module over the fractional ideal $\mathfrak{k}\subset K$ with this ring structure. Assume now that we have a pair of $\mathfrak{k}$-modules $\overline P=(P_1,P_0)$ where $P_0, P_1$ are $\mathfrak{k}$-submodules of the localization $M_{(0)}$ of an $R$-module $M$ and $\mathfrak{k}$ is a fractional ideal in the quotient field $R$  of $K$ as above. Let $x \in \mathfrak{k}$ so that there exists a pair of maps $\overline p=(p_1,p_0), p_1:P_1 \rightarrow P_0, p_0: P_0 \rightarrow P_1$ so that $p_1\circ p_0=(x){\bf 1}_{P_0}$ and $p_0\circ p_1=(x){\bf 1}_{P_1}$ so that we have, after reduction by the ideal $(x)\subset \mathfrak{k}$, a complex $\mathcal{P}(p_1, p_0)$ over $B= \mathfrak{k}/(x)$ as in (\ref{matrixf}) (resp. after application of the functor $\cdot_R \otimes B$). We can alternatively regard $\mathcal{P}(p_1, p_0)\otimes_R B$ as a complex over $R$ resp. $\C[R^*]$ where $\C[R^*]$ is the group ring over $\C$ generated by the (multiplicative) group of units $R^*$ of $R$.
Note that we have here understood that $P_0, P_1$ are modules over the ring $R$ and $R$ is an algebra over $\C$, that is the group ring $\C[R^*]$ acts on the set of finite formal sums $\overline P_0[R^*]$ and $\overline P_1[R^*]$ if $\overline P_0, \overline P_1$ denote reduction of $P_0, P_1$ by $B$, in the classical way. We will denote the resulting complex over $\C[R^*]$ again by $\mathcal{P}(p_1, p_0)$. We denote the cohomology of the image of $\mathcal{P}(p_1, p_0)\otimes_R B$ considered as a complex over $R$ resp. $\C[R^*]$ under $Hom(\cdot, R)$ resp. $Hom(\cdot, \C[R^*])$ by $H^*(\overline P, (B,R))$ resp. $H^*(\overline P, (B,\C[R^*]))$.\\
For the following theorem, we work in the real analytic category, that is we assume that $M$ is real analytic, the maps $s_l, s_2$ are real analytic an $J, J_T \in \mathcal{J}$ are real analytic, expressions like $R=C^\infty(M, \C)$ mean real analytic functions with complex values on $M$ (that is each $f \in R=C^\infty(M, \C)$ can be locally written as a converging power series with complex coefficients). Above we have nearly proven the first assertion of
\begin{theorem}\label{theorem3}
Assume that $\hat s_l, \hat s_2: M\rightarrow \hat P_{G/\tilde G}$ define (cf. Theorem \ref{genclass}) a {\it dual} pair of (generalized) standard irreducible, in general singular Frobenius structures $(\Omega_l, \mathcal{L}_l)$ and $(\Omega_2, \mathcal{L}_2)$ on $(M,\omega)$ with $J, J_T$ associated to $\hat s_l, \hat s_2$ $\omega$-compatible, $\nabla J=0, \nabla J_T=0$, in the sense (and with notation) of Theorem \ref{theorem2} and assume that the $\hat O(n)$-reductions $P^J_L$, $P^{J_T}_L$ associated to $\hat s_l$, $\hat s_2$ are equivalent. Let $\mathcal{C}_{l}, \mathcal{C}_{2}$ the sets of (isolated, but not neccessarily non-degenerated) intersection points of the associated closed sections $s_l, s_2: M\rightarrow T^*M$ with the zero-section of $T^*M$ and $s_l$ being exact with primitive $S:M \rightarrow \R$, where $dS$ is diffeomorphic to the graph of a $C^1$-small Hamiltonian diffeomorphism $\Phi:M\rightarrow M$ as described above Theorem \ref{hopf}. Then with the above notations, there are (well-defined, while non-unique) maps $T_{l,2}:\mathcal{M}^{r, T_{l,2}, J_0}_\infty(T^*_{\C}M, \mathcal{C}_{l})\rightarrow \mathcal{M}^{\tilde r, J_{T,0}}(T_{\C}^*M, \mathcal{C}_{2})$ and $\overline T^*: \mathcal{M}^{\tilde r, \overline T^*_0, J_{T,0}}_\infty(T_{\C}^*M, \mathcal{C}_{2})\rightarrow \mathcal{M}^{r, J_0}_\infty(T^*_{\C}M, \mathcal{C}_{l})$ for appropriate $r \in (\N^+)^k, \tilde r \in (\N^+)^{\tilde k}$, where we further restrict $\overline T^*$ to $\mathcal{M}^{\tilde r, T_{l,2},\tilde T^*_0}_\infty(T_{\C}^*M, \mathcal{C}_{2}):=(\overline T^*)^{-1}(\mathcal{M}^{r, T_{l,2}, J_{0}}_\infty(T^*_{\C}M, \mathcal{C}_{l}))\cap \mathcal{M}^{\tilde r, \overline T^*_0, J_{T,0}}_\infty(T_{\C}^*M, \mathcal{C}_{2})$, arriving at a map $\overline T^*:\mathcal{M}^{\tilde r, T_{l,2}, \overline T^*_0}_\infty(T_{\C}^*M, \mathcal{C}_{2})\rightarrow \mathcal{M}^{\tilde r, T_{l,2}, J_{0}}_\infty(T_{\C}^*M, \mathcal{C}_{l})$ (in the transversal case we have $r= {\bf 1}, \tilde r= {\bf 1}$). Then we have setting $\tilde T_{l,2}:=W_{\tau, \eta} \ast T_{l,2},\ \tau, \eta \in \{0,1\}$ on $M \setminus \mathcal{C}_{l,2}$ that ${\rm im}(\tilde T_{l,2})\subset \mathcal{M}^{\tilde r, T_{l,2}, \overline T^*_0}_\infty(T_{\C}^*M, \mathcal{C}_{2})$ futhermore it holds that $\overline T^*.J=J_{T}$ and we have for the respective cases $\tau, \eta \in \Z_2$ 
\begin{equation}\label{matrix2}
\overline T^* \circ \tilde T_{l,2}= \tilde T_{l,2}\circ \overline T^*= W_{\tau, \eta}\cdot I^{\rm Gr}_n=W_{\tau, \eta} \ast,\quad d \tilde T_{l,2}= d \overline T^*=0,
\quad (\tilde T_{l,2})^*[\cdot, \cdot]_{J_{T, \overline \eta}, \Lambda_{\tilde G}}=[\cdot, \cdot]_{J_{\overline \tau}, \Lambda_{\tilde G}},
\end{equation}
and the latter two conditions are understood to hold on $M\setminus \mathcal{C}_{l,2}$. Further, $W_{\tau, \eta}=\mathcal{B}^{-1}_{\C}(J, J_T,s_2)\frac{(|s_l|^2)^\tau}{(|s_2|^2)^\eta}$, thus $W_{\tau, \eta} \in C_0^\infty(M ,\mathcal{C}_{l,2}, \PI^1)$ for $\eta= 1$ while $W_{\tau, \eta} \in C_0^\infty(M ,\mathcal{C}_{l,2}, \C)$ for $\eta=0$, the former being elements of the quotient field $K_0$ of $R_0=C^\infty_0(M, \C)$ defined by $\mathcal{C}_{l,2}$ as described above. Further $d$ is the exterior derivative induced on ${\rm End}_\omega(T^*M\setminus \pi^{-1}(\mathcal{C}_{l,2}))$ by $d$ acting on $\Omega^1(M, \C)$, interpreted as $d \tilde T_{l,2} \in \Omega ^2(M\setminus \mathcal{C}_{l,2}, TM)$ (see the remark below). Let $\mathcal{H}^*(\tau, \eta)=H^*((\overline T^*,\tilde T_{l,2}), (R_0/W_{\tau, \eta}, R_0))$ be the $\Z_2$-graded cohomology group associated to the matrix factorization $(\overline T^*,\tilde T_{l,2})$. Then we have the following results: 
\begin{enumerate}
\item $\mathcal{H}^*(\tau, \eta)=0$ if $\tau=0$, $\eta=0$ and $\tilde \chi_{J,J_T}=0$ in $H^1(M, \Z)$.
\item If $\tau=1$ or $\eta=1$ or $\tilde \chi_{J,J_T}\neq 0$ we have in general $\mathcal{H}^*(\tau, \eta)\neq 0$ while $\mathcal{H}^*(\tau, \eta)$ is finitely generated over the implied fractional ideals. For the case $(\tau, \eta)=(1,1)$ the module $\mathcal{H}^*(\tau, \eta)$ is non-zero and of finite type while in the case that both $s_l$ and $s_2$ intersect the zero section of $T^*M$ transversally and $\tilde \chi_{J,J_T}=0$, its Euler characteristic $\chi(R_0/W_{1,1})={\rm dim}_{\C}(\mathcal{H}^1(1, 1))- {\rm dim}_{\C}(\mathcal{H}^0(1, 1))$ vanishes.
\item (Riemann-Roch) Assume $M$ is compact, then in general the Euler characteristic $\chi(R_0/W_{1,1})={\rm dim}_{\C}(\mathcal{H}^1(1, 1))- {\rm dim}_{\C}(\mathcal{H}^0(1, 1))$ of the $\Z_2$ graded cohomology $\mathcal{H}^*(\tau, \eta)=H^*((\overline T^*,\tilde T_{l,2}), (R_0/W_{\tau, \eta}, R_0))$ is given by
\[
\chi(R_0/W_{1,1})= \langle \hat \eta_{s_2}, \tilde \chi_{J,J_T}\rangle +\langle \frac{s_2^*c_1(\mathbb{L},\PI^*M)}{\lvert s_l^*c_1(\mathbb{L},\PI^*M)\rvert^2},\mathfrak{c}_1(M, s_l, \overline T_{l,2})\rangle - \langle \frac{s_l^*c_1(\mathbb{L},\PI^*M)}{\lvert s_2^*c_1(\mathbb{L},\PI^*M)\rvert^2},\mathfrak{c}_1(M, s_2,\overline T^*)\rangle,
\]
where $\langle \cdot, \cdot\rangle$ denotes the cohomological Poincar\'e duality pairing defined by $(\omega, J)$ on $M$, thus giving the usual $L^2$-metric on $\Omega^*(M,\C)$ induced by $(\omega,J)$ and $\lvert\cdot\rvert_J$ denotes the induced pointwise norm on $\Lambda^*(T_xM,\C)_x, x \in M$. We note that the first integrand (involving $\hat \eta_{s_2}$) above is understood as being defined over $M\setminus \mathcal{C}_2$ but its integral is shown to converge in an appropriate sense on $M$.
\end{enumerate}
We have the following dichotomy (in the transversal case): if at least one of the Frobenius structures $(\Omega_l, \mathcal{L}_l)$ and $(\Omega_2, \mathcal{L}_2)$ is singular, either the maps $\tilde T^*, \tilde T_{l,2}$ smoothly extend to $\mathcal{M}^{\bf 1}(T^*_{\C}M, \mathcal{C}_{l}\setminus \mathcal{C}_{l}\cap \mathcal{C}_{2})$ resp. $\mathcal{M}^{\bf 1}(T^*_{\C}M, \mathcal{C}_{2}\setminus \mathcal{C}_{l}\cap \mathcal{C}_{2})$ or the map $W$ smoothly extends to a non-singular function $W \in C^\infty(M ,\mathcal{C}_{l,2}\setminus \mathcal{C}_{l}\cap \mathcal{C}_{2}, \C)$ (or none of these two alternatives hold), while in the non-singular case, both alternatives hold.\\
On the other hand, given a map $\hat s_l: M\rightarrow \hat P_{G/\tilde G}$ defining a standard irreducible, in general singular Frobenius structure $(\Omega_l, \mathcal{L}_l)$ being induced by the graph of a $C^1$-small Hamiltonian diffeomorphism on $M$ and a pair of fibrewise linear symplectic vector bundle automorphisms $\tilde T_{l,2}, \tilde T^*$ on $T^*(M \setminus \mathcal{C})$ for some discrete $\mathcal{C} \subset M$ (containing the singular locus of $(\Omega_l, \mathcal{L}_l)$) covering the identity on $M\setminus \mathcal{C}$ and satisfying (\ref{matrix2}), if the Euler characteristic of the associated cohomology group $H^*((\overline T^*,\tilde T_{l,2}), (R_0/W,R_0))$ vanishes, we can construct a unique (generalized) standard, in general singular irreducible Frobenius structure $\hat s_2: M\rightarrow \hat P_{G/\tilde G}$, denoted $(\Omega_2, \mathcal{L}_2)$ so that $\hat s_l, \hat s_2$ define a dual (not necessarily transversal) pair in the sense of Theorem \ref{theorem2} and so that $(\Omega_2, \mathcal{L}_2)$ is singular at most on $\mathcal{C}$. Finally, the latter assignment gives a left-inverse to any choice of the former assignment.
\end{theorem}
{\it Remark.} For an element $T \in {\rm End}_\omega(T^*M)$, we interpret $dT$ as an element of $\Omega^2(M, TM)$ using $dT(s)=d(Ts)-T(ds)$, to be distinguished from interpreting $dT$ as an element of $\Omega^1(M, End_\omega(T^*M))$. The reason for this will be clear in the proof below resp. the subsequent sections, where sections of $\Lambda^*(M, TM)$ will be the central objects. Note that for the Kaehler case, if $\nabla$ is thus torsion-free satisfying $\nabla J=0$ resp. in the non-Kaehler case for the 'deformed' almost complex structure on $M \setminus \mathcal{C}_l$ and an associated symplectic connection $\nabla$ as defined in the proof of Proposition \ref{classi}, we can replace the condition $d \tilde T_{l,2}= d \tilde T^*=0$ by  $\nabla \tilde T_{l,2}= \nabla \tilde T^*=0$ on $M\setminus \mathcal{C}_l$ by arguing similarly as in the proof of Proposition \ref{classi}. Note that in the case of non-Kaehler manifolds $d \tilde T^*_0=0$ does in general not hold for the 'undeformed' $\tilde T^*_0 \in {\rm End}_\omega(T^*M)$ associated to $\hat s^0_l, \hat s^0_2: M\rightarrow \hat P_{G/\tilde G}$ as denoted above, the reason for this being the presence of torsion.
\begin{proof}
Given maps $\hat s_l, \hat s_2: M\rightarrow \hat P_{G/\tilde G}$ that define (cf. Theorem \ref{genclass}) a {\it dual} pair of (generalized) standard irreducible, in general singular Frobenius structures, the well-definedness of $\tilde T_{l,2}$ resp. $\overline T^*$ without the conditions $d \tilde T_{l,2}= d \tilde T^*=0$ and $\tilde T_{l,2})^*[\cdot, \cdot]_{J_{T, \overline \eta}}=[\cdot, \cdot]_{J_{\overline \tau}}$ (in the following called first resp. second integrability condition) was proven in Lemma \ref{ass3} resp. Lemma \ref{affine}, still it is not immediately clear that $\tilde T_{l,2}$ resp. $\overline T^*$ can be chosen so that the two integrability conditions are satisfied. As remarked above, we will in the following denote by $J_0$ resp. $J_{T,0}$ the {\it non-singular} almost complex structures associated to $\hat s^0_l, \hat s^0_2: M\rightarrow \hat P_{G/\tilde G}$ in the sense of Assumption \ref{ass3}, that is before the endpoint-singular homotopy described in Proposition \ref{classi} resp. Theorem \ref{genclass} is applied. Note that restricting $d\hat T_{l,2}=d\tilde T_{l,2}|({\rm span}\ s_l)$ when interpreting $d\tilde T_{l,2}$ as an element in $\Omega^2(M, TM)$, the closedness of $\hat T_{l,2}$ on $\mathcal{M}^r_\infty(T^*_{\C}M, \mathcal{C}_{l})\cap \Gamma(T^*(M\setminus \mathcal{C}_l)\cap ({\rm span}\ s_l))$ follows from the definition of $d\hat T_{l,2}(s)=d(\hat T_{l,2}s)-\hat T_{l,2}(ds), s \in \Omega^1(M\setminus \mathcal{C}_l)$, extending $\hat T_{l,2}$ to $\Lambda^*(M\setminus \mathcal{C}_l)$ by multilinearity and using $ds_l=0$ and $d(T_{l,2}s_l)=ds_2=0$. On the other hand, to achieve the well-definedness of $\tilde T_{l,2}$, we have to multiply $(T^*)^{-1}$ on $T^*M\cap ({\rm span}\ s_l)^\perp$, restricted to appropriate one-dimensional subbundles of $T^*M\cap ({\rm span}\ s_l)^\perp$ by appropriate elements of $\mathfrak{p}(\mathcal{C}_{2})$ resp. by appropriate elements $\mathfrak{p}(\mathcal{C}_{l})$ (or negative powers of it) as we will detail now. We will show that this can be done so that in fact $d\tilde T_{l,2}=0$ and $(\tilde T_{l,2})^*[\cdot, \cdot]_{J_{T, \overline \eta}}=[\cdot, \cdot]_{J_{\overline \tau}}$, while $\tilde T_{l,2}.(\frac{1}{|s_l|^2})^\tau\ast_{J} s_l=(\frac{1}{|s_2|^2})^\eta\ast_{J_{T}} s_2$. For reasons that will be apparent below, we will frequently revert to the language of symplectic reductions in the following (instead of considering symplectic subspaces), in the following thus we will for a coisotropic subspace $W\subset Ann(W)$ often identify the symplectic quotient $W/Ann(W)$ with an appropriate symplectic subspace of a given symplectic space $(V, \omega)$ wrt the projected symplectic form $\omega$ and do not explicitly describe the implicit isomorphism.\\
Consider first the symplectic subbundle $\mathcal{S}_l\subset T^*(M\setminus \mathcal{C}_{l,2})$ given by ${\rm span}(s_l, Js_l)\subset T^*(M\setminus \mathcal{C}_{l,2})$ and denote the restriction $\tilde T_{l,2}|\mathcal{S}_l$ by $\tilde T^0_{l,2}$. By the general Lemma \ref{affine}, $\tilde T^0_{l,2}$ can be chosen to preserve $\omega| \mathcal{S}_l$ but we need a little more so we construct $\tilde T^0_{l,2}$ explicitly. Note that in the situation of Theorem \ref{theorem2}, the first of the conditions (\ref{global}) implies for $s_2:M\rightarrow T^*M$ that $\alpha^-_{J}(s_l^*)=\alpha^-_{J_T}(s_2^*)$ which implies setting $\tilde s_l=\frac{1}{|s_l|^{2\tau}}\ast s_l$ and $\tilde s_2=\frac{1}{|s_2|^{2\eta}}\ast s_2$ that $\omega(\tilde s_l^*, J_{\overline \tau}\tilde s_l^*)=\omega(\tilde s_2^*, J_{T, \overline \eta}\tilde s_2^*)=1$ by the definition of $J, J_{\overline \tau}$ resp. $J_T, J_{T,\overline \eta}$, furthermore in a given local symplectic (unitary) basis $(e_1, J_0e_1)$ of $\mathcal{S}_l$, $\tilde s_l^*$ has the same coordinates as $s_2^*$ in the corresponding basis $(\tilde e_1, J_{T, \eta}\tilde e_1)$ of $\tilde T_{l,2}(\mathcal{S}_l)$, where $\tilde e_1=\tilde T^0_{l,2}.e_1$. From this we can deduce that ${\rm pr}_{\Lambda_{\tilde G, \mathcal{S}_l}}(\tilde s_l)={\rm pr}_{\Lambda_{\tilde G, \mathcal{S}_l}}(\tilde s_2)$, where $\Lambda_{\tilde G,\mathcal{S}_l}\subset \mathcal{S}_l$ is the symplectic reduction of $\Lambda_{\tilde G}$ wrt the coisotropic subspace $S_l=\mathcal{S}_l + \Lambda_{\tilde G}$ in $T^*(M\setminus \mathcal{C}_{l,2})$ ($\perp$ is defined wrt $\omega(\cdot, J_0\cdot)$, the sum is in general non-direct). Analogously ${\rm pr}_{J_{\overline \tau}\Lambda_{\tilde G, \mathcal{S}_l}}(\tilde s_l)={\rm pr}_{J_{T, \overline \eta}\Lambda_{\tilde G, \mathcal{S}_l}}(\tilde s_2)$. By this and setting $\tilde T^0_{l,2}(J_{\overline \tau}\tilde s_l)=J_{T, \overline \eta}(\tilde T^0_{l,2}\tilde s_l)$ finally follows the third equation in (\ref{matrix2}) for the restriction of $[\cdot, \cdot]_{J_{\overline \tau}, \Lambda}$ to $\mathcal{S}_2$, if we require that $\tilde T^*|\mathcal{S}_l$ is defined as the restriction of $\overline T^*$ defined in Lemma \ref{ass3} to $\mathcal{S}_l$ and thus satisfies $Ad(\tilde T^*).J|_{\mathcal{S}_2}=J_T|_{\mathcal{S}_2}$. Note that $\tilde T^0_{l,2}$ defined in the above way preserves $\omega$ and intertwines the (dualized) almost complex structures $J^*_{\overline \tau}$ and $J^*_{T, \overline \eta}$, that is $J^*_{T, \overline \eta}=\tilde T^0_{l,2}\circ J^*_{\overline \tau}\circ \tilde T^0_{l,2}$ on $T^*(M\setminus \mathcal{C}_{l,2})$.\\
We now prove that by the above we get a well-defined and closed map $\hat T^0_{l,2}:\mathcal{M}^{r, \tilde T^0_{l,2}, J_{0}}_\infty(T^*_{\C}M\cap \mathcal{S}_l, \mathcal{C}_{l})\rightarrow \mathcal{M}^{\tilde r, J_{T,0}}_\infty(T_{\C}^*M\cap \mathcal{S}_2, \mathcal{C}_{2})$ with $\mathcal{S}_2\subset T^*(M\setminus \mathcal{C}_{l,2})$ given by $\tilde T^0_{l,2}(\mathcal{S}_l)={\rm span}(s_2, J_{T, \overline \eta}s_2)\subset T^*(M\setminus \mathcal{C}_{l,2})$ for some appropriate pair $r\in (\N^+)^k,\tilde r  \in (\N^+)^{\tilde k}$. For this write $\tilde T^0_{l,2}=\frac{|s_l|^{2\tau}}{|s_2|^{2\eta}}\ast T^0_{l,2}$, where $T^0_{l,2}:\mathcal{M}^{r, T^0_{l,2}, J_0}(T^*_{\C}M\cap \mathcal{S}_l, \mathcal{C}_{l})\rightarrow \mathcal{M}^{\tilde r,J_{T,0}}(T_{\C}^*M\cap  \mathcal{S}_2, \mathcal{C}_{2})$ that is $T^0_{l,2}(s_l)=s_2$ on $M\setminus \mathcal{C}_{l,2}$ for appropriate $r, \tilde r$. Then, if $r, \tilde r$ are so that $s_l|U_i\in\mathfrak{m}^{r_i}_{x_i}\ast_{J_0} \mathcal{E}^0(T^*U_i, U_i\setminus x_i)$ for small nghbhds $U_i$ of $x_i \in \mathcal{C}_l$ resp. $s_2|\tilde U_i\in \mathfrak{m}^{\tilde r_i}_{\tilde x_i}\ast_{J_{T,0}} \mathcal{E}^0(T^*U_i, U_i\setminus \tilde x_i)$ for nghbhds $\tilde U_i$ of $\tilde x_i \in \mathcal{C}_2$ and $s_l, s_2$ are trivial in the respective modules generated by $\mathfrak{m}^{r_i-1}_{x_i}/\mathfrak{m}^{r_i}_{x_i}$ resp. $\mathfrak{m}^{\tilde r_i-1}_{\tilde x_i}/\mathfrak{m}^{\tilde r_i}_{\tilde x_i}$ we have evidently that $T^0_{l,2}\in {\rm End}^r_\omega(T^*M\setminus \pi^{-1}(\mathcal{C}_{l,2}))$ for the above fixed $r \in (\N^+)^k, \tilde r \in (\N^+)^{\tilde k}$ and then the above claim follows, given that $T^0_{l,2}$ is smooth in nghbhds $U_x$ of any $x \in \mathcal{C}_2$ and writing $s \in \Gamma(T^*U_x\cap \mathcal{S}_l, \R)$ as $C^\infty(U_x)$-linear combination of $s_l|U_l$ and $J_0^*s_l|U_l$ we see that $T^0_{l,2}(s)\in \mathcal{M}^{\tilde r}(T_{\C}^*U_x\cap \mathcal{S}_2, \mathcal{C}_{2})$ (multiplication by elements of $C^\infty(U_x)$ does only raise the vanishing order of a section, not lower it). Finally the closedness of $\tilde T^0_{l,2}$ follows from the closedness of $T^0_{l,2}$ and the fact that ${\rm ker}(d|s_l|)={\rm ker}(s_l)$ since $ds_l=0$ and $\nabla J=0$.\\
Before we examine $\tilde T^c_{l,2}:=\tilde T_{l,2}|\mathcal{S}_l^{\perp}$, we note that $\overline T^*: \mathcal{M}^{\tilde r, T^*, J_{T,0}}_\infty(T_{\C}^*M, \mathcal{C}_{2}) \rightarrow \mathcal{M}^{r, J_0}_\infty(T^*_{\C}M, \mathcal{C}_{l})$ can be factorized into the composition $\overline T^*=\tilde T^* \circ |s_2|^2\ast_{J_T}$ with $|s_2|^2\ast_{J_T}:\mathcal{M}^{\tilde r, \overline T^*, J_{T,0}}_\infty(T_{\C}^*M, \mathcal{C}_{2})\rightarrow \mathcal{M}^{\tilde r, \overline T^*, J_{T,0}}(T_{\C}^*M, \mathcal{C}_{2})$ and $\tilde T^*: \mathcal{M}^{\tilde r, \overline T^*, J_{T,0}}(T_{\C}^*M, \mathcal{C}_{2}) \rightarrow \mathcal{M}^{r, J_0}_\infty(T^*_{\C}M, \mathcal{C}_{l})$. Then, over the quotient ring of $R$, $K$, we can define pointwise 'inverse' $(\tilde T^*)^{-1}:= (\tilde T^*)^c \cdot det (\tilde T^*)^{-1}$, where $(\tilde T^*)^c$ is over any $x \in M\setminus \mathcal{C}_l$ the matrix of cofactors of $(\tilde T^*)$, i.e. $(\tilde T^*)^c\cdot (\tilde T^*)=det(\tilde T^*)$. Since $det(\tilde T^*)=det(\overline T^*)$ by the definition of the $R_0$-multiplication $\ast_{J_T}$ in the module $\mathcal{M}^{\tilde r, \overline T^*}_\infty(T_{\C}^*M, \mathcal{C}_{2})$ and by the definition of the matrix of cofactors, we see that with this 'inverse' $(\tilde T^*)^{-1}$ of $\tilde T^*$ over $K_0$ satisfies ${\rm im}((\tilde T^*)^{-1})\subset R_0$ and thus $(\overline T^*)^{-1} \subset \mathcal{M}^{\tilde r, \overline T^*}_\infty(T_{\C}^*M, \mathcal{C}_{2})$.\\
To examine $\tilde T^c_{l,2}=\tilde T_{l,2}|\mathcal{S}_l^{\perp}$ (here $\perp$ refers to the pointwise orthogonal complement wrt $\omega(\cdot, J\cdot)$) consider the symplectic reduction $\mathcal{S}_l^c$ of $T^*(M\setminus \mathcal{C}_{l,2})$ wrt the coisotropic subspace $S^c_l={\rm span}(s_l)+ \mathcal{S}_l^{\perp}$ ($\perp$ here understood as in the previous paragraphs), set $\mathcal{S}_l^0={\rm Ann}(S^c_l)$ and consider the symplectic reduction $\Lambda_{\tilde G, \mathcal{S}^c_l}=(\Lambda_{\tilde G} \cap S^c_l)/\mathcal{S}_l^0$ of $\Lambda_{\tilde G}$ wrt $S^c_l$. We then {\it define}, using the above declarations and for $j \in \{0,1\}$,
\[
\tilde T^c_{l,2}:\mathcal{M}^{r, \mathcal{T}^0_{l,2}}_\infty(T^*_{\C}M\cap \mathcal{S}^c_l, \mathcal{C}_{l})\rightarrow \mathcal{M}^{\tilde r}(T_{\C}^*M\cap \mathcal{S}^c_2, \mathcal{C}_{2})\,\quad \tilde T^c_{l,2}:=(\frac{(|s_l|^2)^{\tau}}{(|s_2|^2)^{\eta}})\ast_{J_{T}}(\overline T^*)^{-1}
\]
over $K_0$, the above considerations show that $\tilde T^c_{l,2}$ has image in $\mathcal{M}^{\tilde r}_\infty(T_{\C}^*M \cap \mathcal{S}^c_2, \mathcal{C}_{2})$. Moreover, $(\overline T^*)^{-1}:\mathcal{M}^{r, T_{l,2}}_\infty(T^*_{\C}M, \mathcal{C}_{l})\rightarrow \mathcal{M}^{\tilde r, T_{l,2},\tilde T^*}_\infty(T_{\C}^*M, \mathcal{C}_{2})$ by the very definitions of the respective $R_0$-modules involved above and thus $\tilde T^c_{l,2}$ is well-defined since the pole order $(\overline T^*)^{-1}$ near $\mathcal{C}_2$ (that is the maximal wrt the natural partial order $\tilde r \in \N^{\tilde k}$ so that ${\rm im}(\overline T^*)^{-1})\subset \mathcal{M}^{\tilde r}_\infty(T_{\C}^*M, \mathcal{C}_{2})$) is determined by $|s_2|^{-2}$.\\
We examine the question of closedness of the thus defined $\tilde T^c_{l,2}$. Assume first that $(M, \omega, J_0)$ is Kaehler and $\nabla$, reduced to $P_{\tilde G}$, is the Levi-Civita connection associated to $(M,\omega, J_0)$ and thus a symplectic torsion free connection satisfying $\nabla J_0=0$, where $J_0$ is the almost complex structure associated to $\hat s^0_l: M\rightarrow \hat P_{G/\tilde G}$ as denoted above. Choose around any $x \in M$ a nghbhd $U\subset M$ and use $\nabla$ to parallel transport a given $O(n)$-frame wrt $(\omega, J, \Lambda_{\tilde G, \mathcal{S}^c_l})$ in $x$, along geodesic curves on $U$ giving an $O(n)$-frame $s_J=(e_1, \dots, e_n,f_1, \dots, f_n):U\rightarrow R^J$ on $U$ so that $(e_1, \dots, e_n)$ spans $\Lambda_{\tilde G, \mathcal{S}^c_l}|U$ (note that parallel transport wrt $\nabla$ preserves $\omega$, $J$ and thus $g_J$ and $\Lambda_{\tilde G, \mathcal{S}^c_l}$) and that coincides at $x$ with the implied Riemannian normal coordinates frame. We then infer that for $\tilde T^*$ at $x \in U$, we can consider $\nabla \tilde T^*(s)=\nabla(\tilde T^*s)-\tilde T^*(\nabla s)$, where $s$ here is a section of the fibre bundle ${\rm Gr}(T^\mathcal{J}_{\C}M^\pm)|U=P_U \times_{Mp(2n,\R)} {\rm Gr}(T^{\mathcal{J}_0}_{\C}\R^{2n})^\pm$, where $P_U=\pi_P^{-1}(U)$ and $\nabla$ is the connection on ${\rm Gr}(T^\mathcal{J}_{\C}M^\pm)$ induced by the symplectic connection on $TM$ associated to $J$ (recall that $J$ is induced by $\hat s_l$), dualized to $T^*M$ using $\omega$. By choosing locally a $d_\nabla$-parallel section $s$ as above, we then infer that $d_\nabla (\tilde T^*|{\rm Im}(\alpha^\pm_J(TM))=0$ in $\Omega^1(M, End_\omega(T_{\C}^*M\cap {\rm Im}(\alpha^\pm_J(TM))))$. Since $\nabla J=0$, we can then infer $d_\nabla (\tilde T^*)=0$ in $\Omega^1(M, End_\omega(T^*M))$. Then antisymmetrization gives the claim $d \tilde T^*=0$ in $\Omega^2(M, TM)$.\\
Note that from the second condition in (\ref{global}) and our assumption on $\hat s_l$ ($s_l$ being exact with primitive $S:M \rightarrow \R$, where $dS$ is diffeomorphic to the graph of a $C^1$-small Hamiltonian diffeomorphism $\Phi:M\rightarrow M$) it follows immediately that if $\hat s_l, \hat s_2$ define a dual pair in the sense of Theorem \ref{theorem2}, then $[\cdot, \cdot]_J|{\rm span}(s_l)$ is preserved by the pair $(\tilde T_{l,2}, \tilde T^*)$ in the sense of the second integrability condition of (\ref{global}). We now claim that if we choose the elements $(g^j_i)$ constructed above so that locally the forms $d(g^j_i) \in \Omega^1(U_i)$ satisfy the second condition of (\ref{global}), then $T_{l,2}$ as constructed above satisfies indeed $(\tilde T_{l,2})^*[\cdot, \cdot]_{\tilde T^*.J}=[\cdot, \cdot]_J$, but this follows essentially from the definitions.\\
We sketch the proofs of the assertions (1)-(3) in the Theorem, beginning with (3). Considering $s_2=\tilde T^0_{l,2}.s_l$ and the fact that over any appropriate $U\subset M$, we can write (cf. \cite{gilkey}, Chapter 2.3) 
\[
\mathfrak{c}_1(M, s_l, \overline T_{l,2})|\mathcal{S}_l=s_l^*\frac{1}{2\pi i}(T^0_{l,2}.dz_J)\wedge (T^0_{l,2}.d\overline z_J)/(1+r^2)^2,\ r=\|{z}_J\|,
\]
where $z_J$ is an appropriate projective coordinate on the complex one dimensional bundle $\pi^{-1}(U)\cap \mathcal{S}_l$ (we choose a coordinate system $\Phi_U: \pi^{-1}(U) \rightarrow U \times \C^{n+1}$ on $\pi^{-1}(U)$ so that $\Phi_U(s(x))=(x, r(x),0, \dots, 0), x \in U$ for some positive function $r: U\rightarrow \R^+$), while analogously $s_2^*(c_1(\mathbb{L},\PI^*M)|\mathcal{S}_l)=s_l^*\frac{1}{2\pi i}\tilde (T^0_{l,2})^*(dz_J\wedge d\overline z_J)/(1+r^2)^2$ and the fact that we can choose $\tilde T^0_{l,2}|\mathcal{S}_l$ so that $\lvert\tilde T^0_{l,2}|\mathcal{S}_l\rvert_J=1$, we can infer
\[
\langle s_2^*c_1(\mathbb{L},\PI^*M),\mathfrak{c}_1(M, s_l, \overline T_{l,2})\rangle= {\rm det}\left[d_{\nabla}(\tilde T^0_{l,2}|\mathcal{S}_l).(\tilde T^0_{l,2})^{-1}|\mathcal{S}_l)\right] \left\lvert s_l^*c_1(\mathbb{L},\PI^*M)\right\rvert^2
\]
where the inverse $(\tilde T^0_{l,2})^{-1}$ is taken over the quotient field $K_0$ of $R_0$ and is here interpreted as an endomorphism of $T^v(T^*U)\simeq T^*U$. We can then invoke well-known formulas on abstract residues (cf. J. Tate, \cite{tate}) to evaluate the right hand side above, this then entails a localization around the critical points of $s_l$, in this case (since $\tilde T^0_{l,2}$ is singular on $\mathcal{C}_l$), analogous arguments hold for the other terms in Assertion (3), where we localize around the elements of $\mathcal{C}_2$. 
The idea here is that each 'closed' point $p \in M$ defines a completion $\hat K_p$ of the ring $K_0:=C^\infty_0(M, \mathcal{C}_{l,2}, \PI)$ (note we slightly shrink $K_0$ relative to the previous notation) by looking at the $\mathfrak{m}$-adic topology defined by the maximal ideal $\mathfrak{m}$ of $p$. $K_0$ naturally embeds into $\hat K_p$ for any $p \in M$, we denote the image of $K_0 \hookrightarrow \hat K_p$ by $K_p$. Analogously we can look at the completions $\hat R_p$ of the rings $R_0$ by $\mathfrak{m}$-adic topology associated too $p \in M$, $\hat R_p$ embeds naturally into $\hat K_p$ and $R_0$ embeds into $\hat R_p$, the image of the latter we denote by $R_p$. We can consider for any finite subset $S\subset M$ (in practice, only $S\subset \mathcal{C}_{l,2}$ will be proven to contribute) the intersection $R_S=\cap_{p\in S} R_p\subset R_0$. We can then define
\[
\hat R_S=\prod_{s \in S} \hat R_p,\ \hat K_S=\prod_{s \in S} \hat K_p
\]
We may then argue that $\hat K_S/(K_0+\hat R_S)$ is finite dimensional over $\C$ and analogously to the arguments in (\cite{tate}) this then implies that for the 'residue map' $res^{K_0}_A:\Omega ^1_{K_0/\C}\rightarrow \C$ on the module $\hat K_S$ over $K_0$ which is as defined in loc. cit. over any $\C$-subspace $A \subset \hat K_S$ it holds that
\[
res^{K_0}_{R_S}(\omega)=res^{\hat K_S}_{\hat R_S}(\omega)=\sum_{p \in S} res^{\hat K_p}_{\hat R_p}(\omega).
\]
Idnetifying this abstract residue map with our above analytic considerations, this 'algebraic localization' method in turn leads to the analytic localization one needs to prove the Assertion (3).\\
To prove that $\chi(R_0/W_{1,1})=0$ in the transversal case and if $\tilde \chi_{J,J_T}=0$, we can represent the class $s_2^*c_1(\mathbb{L},\PI^*M) - s_l^*c_1(\mathbb{L},\PI^*M) \in H^2(M, \Z)$ by an appropriate Borel-Moore cycle, this then entails the assertion using the functoriality of Borel-Moore homology under appropriate continuation mappings derived from the Morse theory of the two functions underlying $s_l$ resp. $s_2$, alternatively we may relate the two residue terms in (3.) which were discussed above in the transveral case to the Euler chracteristic of $M$, this then gives the result by the invariance of the latter, the details will appear in a continuation of this article.
\end{proof}

\end{document}